\documentclass[10pt]{article}
\usepackage{amsmath, amsfonts, amsthm, amssymb, verbatim,dsfont}
\usepackage{hyperref}
\usepackage{graphicx}
\usepackage[mathcal]{euscript}
\usepackage[applemac]{inputenc}
\usepackage{dsfont} 
\usepackage{mathrsfs}
\usepackage[usenames]{color}
\usepackage{tikz}
\usetikzlibrary{decorations.pathreplacing}

\usepackage[T1]{fontenc}%
\theoremstyle{plain}
        \newtheorem{theorem}{Theorem}[section]
        \newtheorem{proposition}[theorem]{Proposition}
        \newtheorem{lemma}[theorem]{Lemma}
        \newtheorem{corollary}[theorem]{Corollary}
\theoremstyle{definition}
        \newtheorem{definition}[theorem]{Definition}
        \newtheorem{remark}[theorem]{Remark}
\theoremstyle{plain}
        
\numberwithin{equation}{section}

\newcommand \be           {\begin{equation}}
\newcommand \ee            {\end{equation}}

\newcommand \Pbold           {\mathbf{P}} 
\newcommand \PP \Pbold

\newcommand \del           \partial
\newcommand \eps            \epsilon

\newcommand\ud{\, \mathrm {d}}

\DeclareMathOperator    \dive {\nabla\cdot}
\newcommand{\cref}{c_{\mathrm{ref}}}
\newcommand{\nn}{{\mathbf{n}}}

\def\XXint#1#2#3{{\setbox0=\hbox{$#1{#2#3}{\int_{\Omega}}$}
\vcenter{\hbox{$#2#3$}}\kern-.5\wd0}}

\let\oldmarginpar\marginpar
\renewcommand\marginpar[1]{\-\oldmarginpar[\raggedleft\footnotesize #1]%
{\raggedright\footnotesize #1}}
\def\build#1_#2^#3{\mathrel{
\mathop{\kern 0pt#1}\limits_{#2}^{#3}}}
\allowdisplaybreaks[1]

\begin{document}

%\title{Well-posedness for a chemotaxis system modeling ant foraging}
\title{Analysis of a chemotaxis system modeling ant foraging}
%\title{Analysis of a chemotaxis system modeling trail formation in ant foraging}

\author{Ricardo Alonso$^1$, Paulo Amorim$^2$ and Thierry Goudon$^3$}
\footnotetext[1]{
Department of Mathematics, PUC-Rio, Rua Marqu\^es de S\~ao Vicente, 225 -Gavea, RJ - Brasil.  Email: {\tt ralonso@mat.puc-rio.br}}
\footnotetext[2]{
Instituto de Matem\'atica, Universidade Federal do Rio de Janeiro,
Av. Athos da Silveira Ramos 149,
Centro de Tecnologia - Bloco C,
Cidade Universit\'aria - Ilha do Fund\~ ao,
Caixa Postal 68530, 21941-909 Rio de Janeiro,
RJ - Brasil. Email: {\tt paulo@im.ufrj.br}. Web page: \url{http://www.im.ufrj.br/~paulo/}}
\footnotetext[3]{Inria,  Sophia Antipolis M\'editerran\'ee Research Centre, Project COFFEE
\& Univ. Nice Sophia Antipolis, CNRS, Labo. J.-A. Dieudonn\'e, UMR 7351,
Parc Valrose, F-06108 Nice, France. Email: {\tt thierry.goudon@inria.fr}}

\date{}

\maketitle 

%\tableofcontents

\begin{abstract} 
In this paper we analyze a 
 system
of PDEs recently introduced
in [P. Amorim, {\it Modeling ant foraging: a {chemotaxis} approach with pheromones and trail formation}], in order  to describe the dynamics of ant foraging.
The system is made of convection-diffusion-reaction equations, and the coupling is driven by chemotaxis mechanisms.
  We establish the well-posedness for the model, 
  and investigate the regularity issue for a large class of integrable data.
  Our main focus is on the (physically relevant) two-dimensional case with boundary conditions, where we prove that the solutions remain bounded for all times. 
  The proof  involves a series of fine \emph{a priori} estimates in Lebesgue spaces.  
\end{abstract}

{\small\noindent
\\[.3cm]

\noindent{\bf Math.~Subject Classification.}  92D25, %Pop. dynamics,
92D50   %	Animal behavior
92C17,   	%Cell movement (chemotaxis, etc.)
35K55   %	Nonlinear parabolic equations
35B65   %	Smoothness and regularity of solutions
}
%\tableofcontents

\section{Introduction}
One of the most interesting phenomena arising in the collective behavior of ants is the formation of trails. Indeed, while each individual ant has a very limited cognitive ability, the population as a whole is capable of complex, organized collective behavior, such as brood rearing, waste management or fungus gardening. Even more striking is the fact that many of these activities, especially trail formation (which occurs during foraging, migration, or aggression), are essentially leaderless and yet highly organized. 

Many tools have evolved in ant societies to allow for this sort of complex, so-called \emph{emergent} behavior. One of the most important is the use of pheromones as a means of communication between individuals. Pheromones are chemical compounds secreted by ants (and many other species of animals) which are used to convey information between individuals and to signal different states, such as presence of food or alarm. Each pheromone triggers a corresponding behavior in individuals:  when an ant detects alarm pheromone, it becomes itself alarmed and secretes more pheromone, leading to a chain reaction among the population, whose effect is to elicit an apparently organized defense response. 

We are interested here in the trail-following behavior of ants, which is triggered in part by trails  of pheromones. When an individual ant, foraging at random, encounters a food source, it typically travels back to the nest leaving on the substrate a trail of pheromone. When other ants encounter the trail, they follow it in the direction of the food, and upon finding the food, they head back to the nest and deposit more pheromone. Thus, as long as food is available, the trail will be reinforced and the food will be removed efficiently. Conversely, when the food is exhausted, the evaporation and diffusion of the pheromone quickly erase the trail, when it is no longer being reinforced.

The dynamics of ant foraging behavior has recently come under increased interest from mathematicians trying to find suitable frameworks in which to analyze this behavior. We refer the reader  to the recent works \cite{Amorim,Bertozzi2014,Boissard2012,Ryan}, among others. In \cite{Amorim} (see also the independent work \cite{Bertozzi2014}), 
a system of PDEs, see  \eqref{1000} below, is introduced in order to describe the dynamics of ant foraging.
Roughly speaking, the population splits into two parts: the ants searching for food, and the ants going back to the nest, and the pheromone production is interpreted as a chemotactic mechanism 
that drives the population to privileged 
directions. The discussion on the modeling 
issues in \cite{Amorim} is complemented by a set of convincing numerical simulations 
that illustrate the ability  
of the model to reproduce relevant  behaviors of ant populations.
Here, we aim at analysing the mathematical properties of these equations.

More generally, finding relevant models able to reproduce the formation of  the space-time heterogeneous patterns
 observed in life sciences is 
becoming a very active field, particularly motivated by the landmark contributions of  T. Vicsek {\it et al} \cite{Vic}, and F. Cucker and S. Smale \cite{CuS2, CuS1}
about the formation of flocks in large populations of birds or fish.
The key feature relies on  the transmission by  the individuals of the information contained in their close  environment, so that the
 whole population organizes according to remarkable patterns.
 In this vein, various models have been introduced, which have led to
 original problems for mathematical analysis and  fascinating numerical simulations that reproduce certain features of natural phenomena; we refer the reader
 to the surveys 
 \cite{Coll, Vic2} for an overview on the subject.
Here, following \cite{Amorim}, we are interested in continuum models, where populations are described by their local concentrations.
The interaction between the individuals can be though of through a certain potential, which is defined self--consistently, thus depending on the variations of the 
concentrations.  
This is reminiscent to the theory of chemotaxis, which dates back  to C.~Patlak \cite{Patlak} and E.~Keller--L.~Segel \cite{KellerSegel70,KellerSegel71} to   model
the behavior of certain bacteria and slime molds, which are attracted to chemical substances that they themselves secrete. 
In particular, the possible formation  in finite time of singularities in the solutions of the Keller--Segel equations, the concentration into Dirac masses corresponding to the aggregation of the population into a single location, has motivated a huge amount of mathematical works, see for instance \cite{GA, JL, RZ}. 
By now, the mathematical theory of the Keller--Segel system 
 is well established, and we refer the reader to the surveys \cite{HillenPainter2009,Horstmann1,Horstmann2} for further information and references.
 To describe interaction mechanisms between living organisms by chemotactic principles has been successfully 
 adapted to many different situations, see e.~g.~\cite{Kol1,Fasano,Kol2,GNRR,Horst2,MKB,TW}  to mention  a few.

To be more specific, in this paper, we study the basic mathematical theory for nonnegative solutions $(t,x)\mapsto (u,w,p,c)(t,x)$ of the following model for ant foraging
\be
\label{1000}\tag{\bf {SPD}}
\left\{
\begin{aligned}
&\partial_t u -  \Delta u + \dive\big( u\, \chi\nabla p  \big) = - uc + w N,
\\[.3cm]
& \partial_t w - D_w \Delta w + \dive\big( w\,  {\nabla v}   \big) =  uc - w N,
\\[.3cm]
& \partial_t p -  D_p \Delta p = P w - \delta p,
\\[.3cm]
& \partial_t c = - u\, c.
%\\[.3cm]
%&+ \text{Appropriate boundary conditions and initial data.}
\end{aligned} \right.
\ee
In this system of PDEs, the unknowns are
\begin{itemize}
\item[$\bullet$] the density of foraging ants $u$,
\item[$\bullet$] the density of returning ants $w$:
it describes the  ants which have found  food and are returning to the nest,
\item[$\bullet$] the concentration of the pheromone $p$,
\item[$\bullet$] the distribution of the food $c$.
\end{itemize}
These nonnegative quantities depend on  time $(t\geq 0)$ and space $(x\in\Omega\subset \mathbb R^n)$ variables.
The data of the problem are
\begin{itemize}
\item[$\bullet$] the site of the nest, embodied into the function $x\mapsto N(x)$,
\item[$\bullet$] a function $x \mapsto P(x)$ that describes the  pheromone deposition as returning ants approach the nest; typically this function decreases as the distance to
the nest decreases,
\item[$\bullet$] a nest-bound vector field $x\mapsto \nabla v(x)$ representing the speed of the ants when returning to the nest
(it might contain information on the topography, obstacles...),
\item[$\bullet$] Diffusion, sensitivity and evaporation coefficients $D_w,D_p, \chi,$ and $\delta$, which  are all positive constants.
\end{itemize}
The system \eqref{1000} will be addressed in the sequel as the \textbf{Slow  Pheromone Diffusion}  model as the pheromone diffusion time scale is comparable to that of the dynamics of the ant foraging.
We refer the reader to \cite{Amorim} for  details on the biological motivation for system \eqref{1000}.
In what follows, we will take $P\equiv 1$ and  $D_p=D_w= \chi =1$ for the sake of simplicity. 
We will study the system \eqref{1000} on the whole Euclidean spatial domain $\mathbb{R}^{2}$  with a few comments applying to  dimension $n=3$.

We will also work with the simplified situation where the pheromone diffusion time scale is small compared to the dispersal of the ants. In addition, we suppose that there exists a very abundant, or renewable, food source $0\leq c:=c(x),$ so that we can assume it is a given function of space.
These simplifying assumptions lead us to the following reduced system, with unknowns
$(t,x)\mapsto (u,w,p)(t,x)$
\be
\label{2000}\tag{\bf{FPD}}
\left\{
\begin{aligned}
&\partial_t u -  \Delta u + \dive\big( u\, \nabla p \big) = - uc + w N,
\\[.3cm]
& \partial_t w -  \Delta w + \dive\big( w\,  {\nabla v}   \big) =  uc - w N,
\\[.3cm]
& -  \Delta p =  w - \delta p .
\end{aligned} \right.
\ee
 It will be referred to  as the \textbf{Fast Pheromone Diffusion}  model.
This system will be considered in a  domain $\Omega\subset\mathbb{R}^{2}$ having a smooth boundary $\del\Omega$. 
In order to conserve  the total mass of ants, we impose the following zero-flux boundary conditions
\be
\label{2500}
\aligned
 \nabla u \cdot \nn\Big|_{\partial\Omega} = (w\nabla v- \nabla w) \cdot \nn\Big|_{\del\Omega}  = \nabla p \cdot \nn \Big|_{\del\Omega} = 0,
\endaligned
\ee
where $\nn$ stands for the outward unit normal vector to $\del\Omega$.  
While this is not crucial for most of the analysis, we can bear in mind the fact that 
physically relevant velocity fields satisfy
\begin{equation}\label{nv+}
{\nabla v} \cdot \nn\big|_{\del\Omega} \leq  0\end{equation} since $\nabla v$ is  pointing towards the nest.  Assumption \eqref{nv+} will play a role when proving uniform propagation of the $L^{\infty}$--norm for solutions, namely, estimate \eqref{Bound1} and later, for the local existence of classical solutions.  Therefore,  ants do not escape the domain $\Omega$.   The initial data 
\be
\label{CIFPD}
u\Big|_{t=0}=u_o,\qquad w\Big|_{t=0}=w_o
\ee
will be assumed, naturally, as nonnegative integrable functions.  As a consequence, we have
\be
\label{MCFPD}
\displaystyle\int_\Omega (u+w)(t,x)\ud x=\displaystyle\int_\Omega (u_o+w_o)(x)\ud x.
\ee
We shall 
prove the global well-posedness in dimension $n=2$ for
both models \eqref{1000} and \eqref{2000}. In fact, weak solutions of such systems are bounded
% and enjoy Sobolev regularity 
in $(0,T]\times\Omega$ for any $T>0$.  
This is in contrast with the situation known for the usual  
Keller--Segel system
\be
\label{KS}\left\{
\aligned
&\partial_t \rho+\dive(\rho\nabla\Phi)-\Delta \rho=0,
\\[.3cm]
&-\Delta \Phi=\rho.
\endaligned\right.
\ee
It is indeed well--known that, if the initial mass $\int \rho(0,x)\ud x$ exceeds a certain threshold (which depends on the domain and the space dimension $n\geq 2$),
then the solution of \eqref{KS} blows up: $\|\rho(t,x)\|_\infty \to\infty$ as $t\to T^\star<\infty$, typically exhibiting a concentration to a certain location.
Roughly speaking, this effect can be seen as a result of the competing effects between diffusion, and the explosive behavior of the ODE $\dot y=y^2$ (which is obtained neglecting diffusion and looking
at the solution $\rho$ evaluated along the characteristics associated to the field $\nabla\Phi$).
At first sight, the structure of the systems  \eqref{1000} and \eqref{2000} is quite close to that of the Keller--Segel system \eqref{KS} 
and one may wonder whether of not such threshold phenomena occur to produce the blow up of the solutions.
However, we shall see that  the introduction of an additional population (the returning ant population $w$), which is itself 
subjected to a regularizing  parabolic equation, prevents the blow up formation.

The paper is organized as follows.  In Section \ref{2} we set up  the assumptions necessary to the analysis  and we state the main results. 
  Section \ref{3} is devoted to  the model \eqref{2000} and it is divided in four subsections for clarity.  The first and second subsections present the core of the analysis where increasing integrability of the solutions is shown by means of \textit{a priori} estimates.  The initial step is to prove instantaneous generation of $L^{\gamma}$-integrability. To be more specific, 
  we prove that $L^{1}$ initial data will lead to solutions lying in
$L^\infty\big(t_{\star},T;L^{\gamma}(\Omega)\big)$  for \text{any} $\gamma\in(1,\infty)$ and any positive times $0<t_{\star}\leq T\leq \infty$.  The bounds given for the $L^{\gamma}$-norm of the solution will depend on the \textit{structure of the system}, that is, on the data $N$, $c$, $v$ and more importantly, on the conserved quantity, that is 
   the initial mass of the total ant population $m_o=\int u_o\ud x+ \int w_o\ud x$.  Such bounds are independent of the time existence interval, a fact that can be used to
   justify that the solutions are globally defined.  Furthermore, it is  possible  to use such a result  to prove that solutions are actually uniformly bounded for any positive time. The proof of this
    fact relies on the De Giorgi energy level set method, \cite{DeG}.  The third subsection is concerned with 
%    basic Sobolev regularity.
%    It will allow us to establish  
    the global existence of classical solutions for the problem.  
     Given smooth initial data,  local in time existence  of solutions can be justified by using the classical Banach fixed-point theorem on suitable metric spaces. Next, the global \textit{a priori} estimates which are valid for classical solutions allow us to justify that the lifespan of these solutions is actually infinite. 
      We finish Section~\ref{3} by presenting a global well-posedness theorem  for weak solutions of \eqref{2000}
     with initial data in $L^{1}\cap L^{\gamma}$ with $\gamma>2$.  The proof uses approximation by the classical solutions just found.  Finally, in Section \ref{4} we  discuss  the well-posedness of the system \eqref{1000}. Our approach  uses elementary properties of the heat kernel which is, somehow, a different approach than the one used for the analysis of     \eqref{2000}.

\section{Notations, hypotheses and main results}\label{2}
The main results of this paper are concerned with the well-posedness of weak solutions of the  system \eqref{2000}.
\begin{definition}\label{dws}
We say that the triple $(u,w,p)$ is a \emph{weak solution} of  the system \eqref{2000} if it satisfies:
\begin{itemize}
\item [\it(i)] $(u,w)\in L^{2}\big(0,T; H^{1}(\Omega)\big)$ and $(\partial_t u ,\partial_t w )\in L^{2}\big(0,T; H^{1}(\Omega)^{\star}\big)$.
\item [\it(ii)] Equations \eqref{2000} are solved in the %$L^{2}\big([0,T); H^{1}(\Omega)^{\star}\big)$ 
sense that, for any test function $\zeta\in C^\infty([0,\infty)\times\Omega)$, compactly supported in $[0,T)\times \overline\Omega$,
we have
\[\begin{array}{l}
{
\displaystyle\int_0^T\displaystyle\int_\Omega \big(-u\partial_t\zeta}
+(\nabla u-u\nabla p )  \cdot\nabla\zeta  \big)(t,x)\ud x\ud t- \displaystyle\int_\Omega u_o(x)\zeta(0,x)\ud x
\\[.3cm]
\qquad\qquad=\displaystyle\int_0^T\displaystyle\int_\Omega (-uc+wN)\zeta(t,x)\ud x\ud t,
\\[.3cm]
{\displaystyle\int_0^T\displaystyle\int_\Omega \big(-w\partial_t\zeta}
+(\nabla w-w\nabla v )   \cdot\nabla\zeta  \big)(t,x)\ud x\ud t- \displaystyle\int_\Omega w_o(x)\zeta(0,x)\ud x
\\[.3cm]
\qquad\qquad=\displaystyle\int_0^T\displaystyle\int_\Omega (uc-wN)\zeta(t,x)\ud x\ud t,
\\[.3cm]
\displaystyle\int_\Omega \nabla p\cdot\nabla \zeta\ud x=\displaystyle\int_\Omega (w-\delta p)\zeta\ud x.
\end{array}\]
%\item [\it(iii)] And, $(u(0),w(0))=(u_o, w_o)$ almost everywhere.
\end{itemize}
\end{definition}
\noindent
\textbf{Main hypotheses.} The following assumptions will  be used throughout this Section:
\begin{itemize}
\item [\it(i)] The initial data $(u_o, w_o)$ is nonnegative with finite mass:
\begin{equation}\label{FinMass}
\int_{\Omega} u_o(x) 
 + \int_{\Omega} w_o(x) \ud x = m_o<\infty\,.
\end{equation}
\item [\it(ii)] The parameters of the  system  \eqref{2000}  are such that
\begin{equation}\label{HypH}\tag{\bf H}
\big(N,\,c,\, \nabla v,\, \Delta v\big) \in L^{\infty}(\Omega).
\end{equation}
The results can be extended under weaker assumptions on $v$ at the price of  more involved technicalities  in the estimates.  As a convention, let us agree here that when a constant is referred to depend on \eqref{HypH} it means that this constant  depends on the $L^{\infty}$-norms of $N$, $c$, $\nabla v$ and $\Delta v$.
\item [\it(iii)] The domain $\Omega$ is of class $C^{2}$.
\end{itemize}
In what follows we will use the shorthand notation $\gamma^{\pm}$ to denote a number close but strictly bigger/smaller that $\gamma$.  
Having all these in mind let us gather the main results regarding  \eqref{2000} in one single statement.
\begin{theorem}\label{TFINAL}
Let $\delta>0$ be fixed. Let  $(u_o,w_o)$ be a  pair of nonnegative functions in $L^{1}\cap L^{2^{+}}(\Omega)$.  Then, there exists a unique nonnegative weak solution to  \eqref{2000}.  In the case $\delta=0$ uniqueness continues holding up to a constant in the pheromone $p$ distribution.
Furthermore, the following estimates are satisfied by the solutions for any $0<t\leq T<\infty$:
\begin{itemize}
\item [\it(i)] $L^{\gamma}$-integrability 
\begin{equation*}
\|w(t)\|_{\gamma} + \|u(t)\|_{\gamma} \leq C(m_o,\gamma)\Big(1 + \frac{1}{t^{(1/\gamma')^{+}}}\Big)\,,\quad\text{for any $ \gamma\in[1,\infty)$},
\end{equation*}
\item [\it(ii)] $L^{\infty}$-bound
\begin{equation*}
\|w(t)\|_{\infty} + \|u(t)\|_{\infty} \leq C(m_o)\Big(1 + \frac{1}{t^{1^{+}}}\Big)\,.
\end{equation*}
The constants in $(i)$ and $(ii)$ depend  on \eqref{HypH} but do not depend on $T$.
\item [\it(iii)] If, in addition, the initial data are in $L^\gamma(\Omega)$ for $\gamma \in (1,\infty]$, then the previous estimates are improved to
\begin{equation*}
\|w(t)\|_{\gamma^+} + \|u(t)\|_{\gamma} \leq C'(m_o,\gamma),
\end{equation*}
\end{itemize}
where the constant $C'$ depends on \eqref{HypH} and on $\|w_o\|_{\gamma^+}, \|u_o\|_{\gamma}$ but not on $T$.
\end{theorem}
\begin{remark}
Once the $L^\infty$ estimate has been established, it can be used to investigate further regularity of the solution.
In particular, it implies that $\nabla u$ and $\nabla p$ are bounded
on $[t_\star,T]\times \Omega'$, for any $0<t_\star<T$ and any domain $\Omega'$ strictly included in $ \Omega$;
see  \cite[Th. VII.6.1]{LUS}.
Assuming that $N, c, v$ are $C^\infty$, we can boil down 
 a bootstrap argument to show that the solution $(u,w)$ is actually of class $C^\infty$ in any such subdomain $[t_\star,T]\times \Omega'$, see for instance \cite[Prop. A.1 \& Th. A. 1]{GV}.
\end{remark}
\section{Analysis of the  model \eqref{2000}}
\label{3}%
In this section we provide a series of \textit{a priori} estimates to build up enough regularity to prove the existence of classical solutions of
  \eqref{2000}.
\begin{definition} A \emph{classical solution} $(u,w,p)$ of the system \eqref{2000} is defined as a triple $(u,w,p)$ satisfying the following:
\begin{itemize}
\item [\it(i)] The triple $(u,w,p)\in C\big([0,T];L^{2}(\Omega)\big)$ and each of the terms in the  system  \eqref{2000}  (i.e. $\partial_t u ,\, \Delta u,\, \nabla\cdot(u\,\nabla p),$ and so forth) are well defined functions in $L^{2}\big((0,T)\times\Omega\big)$,
\item [\it(ii)] The equations \eqref{2000} are satisfied almost everywhere,
\item [\it(iii)] The initial data $(u,w)\big|_{t=0}=(u_o,w_o)$ and the boundary condition \eqref{2500} are satisfied almost everywhere.  
\end{itemize}
\end{definition}
\subsection{From $L^{1}$ to $L^{\gamma}$ regularity.}
In this section we prove instantaneous generation of $L^{\gamma}$ integrability, with $\gamma>1$, for nonnegative classical solutions of \eqref{2000} 
 associated to  initial data lying in $L^{1}$ only.  A crucial fact used in the argument is mass conservation, that is, classical solutions $\big(u(t),w(t)\big)$ satisfy \eqref{MCFPD} that we rewrite as
\begin{equation*}
\int_{\Omega} u(t,x)\ud x + \int_{\Omega} w(t,x)\ud x=\int_{\Omega} u_o(x)\ud x+ \int_{\Omega} w_o(x)\ud x=m_o \;\; \text{ for any }\; t>0.
\end{equation*}
\begin{proposition}\label{T1000}
Let $(u,w)$ be a classical nonnegative solution of the   system \eqref{2000} with boundary conditions \eqref{2500}.  Then, for any $\gamma\in[1,\infty)$ we have the estimate
\begin{equation*}
\|w(t)\|_{\gamma}  + \|u(t)\|_{\gamma} \le C(m_o,\gamma)\Big(1+\frac{1}{t^{(1/\gamma')^{+}}}\Big)\,,\quad t>0\,,
\end{equation*}
where the  constant $C(m_o,\gamma)$ depends additionally on \eqref{HypH}, but, is independent of time.  Furthermore, for any $\gamma\in[1,\infty)$ this estimate can be upgraded, provided we add the dependence of the integrability of the initial data to the constant, to
\begin{equation*}
\int_{\Omega} w^{\gamma^{+}} (t)  \, \ud x + \int_{\Omega} u^\gamma (t)  \ud x\leq C(m_o, \|u_o\|_{\gamma}, \|w_o\|_{\gamma^{+}})\,,\quad t>0\,.
\end{equation*}
The constant depends  on \eqref{HypH} (but is independent of time).
\end{proposition}

\proof
We start by multiplying equation \eqref{2000} by $ u^{\gamma-1}$, with $\gamma>1$, and integrating with respect to the space variable. We obtain
\begin{equation*}
\aligned
& \frac{1}{\gamma} \frac{\ud}{\ud t} \int_{\Omega} u^\gamma   \ud x - \frac{\gamma-1}{\gamma}\int_{\Omega} \nabla p \cdot \nabla (u^{\gamma})   \ud x
+ (\gamma-1) \int_{\Omega} u^{\gamma-2} |\nabla u|^2   \ud x
\\[.3cm]
& \qquad\qquad \qquad\qquad = - \int_{\Omega} u^\gamma c  \ud x + \int_{\Omega} N w u^{\gamma-1}  \ud x\,.
\endaligned
\end{equation*}
Now multiply the third (pheromone) equation by $u^\gamma$ and integrate by parts to conclude that
\begin{equation*}
\aligned
&  \int_{\Omega} \nabla p \cdot \nabla (u^{\gamma})  \, \ud x
 \le  \int_{\Omega} w u^\gamma  \ud x\,.
\endaligned
\end{equation*}
Next using
\begin{equation*}
\aligned
\int_{\Omega} u^{\gamma-2} |\nabla u|^2 \, \ud x
=  \frac{4}{\gamma^2}\int_{\Omega} |\nabla u^{\gamma/2}|^2  \ud x\,,
\endaligned
\end{equation*}
and the Young inequality with conjugate exponents $p=\gamma$,  $p'=\frac\gamma{\gamma-1}$, and
$p=\gamma+1$, $p'=\frac{\gamma+1}{\gamma}$, respectively,
we find 
\be\label{3000}\begin{array}{l}
\displaystyle\frac{\ud}{\ud t} \int_{\Omega} u^\gamma   \ud x
+  4 \frac{\gamma-1}{\gamma}\int_{\Omega} |\nabla u^{\gamma/2}|^2  \ud x
\\[.3cm]
\qquad\le \displaystyle \gamma \int_{\Omega}  N w u^{\gamma-1}   \ud x +  (\gamma-1) \int_{\Omega}   w u^{\gamma}  \ud x\\[.3cm]
\qquad \displaystyle
\leq \|N\|_{\infty} \int_{\Omega} w^{\gamma}  \ud x +  \gamma\,\|N\|_{\infty} \int_{\Omega} u^{\gamma}  \ud x   \\[.3cm]
\qquad\qquad+ \displaystyle\frac{1}{\varepsilon}\int_{\Omega}   w^{\gamma+1}  \ud x + \gamma\varepsilon^{1/\gamma}\int_{\Omega} u^{\gamma+1}  \ud x\,.
\end{array}\ee
We have also used weighted Young's inequality $$ab=\frac{a}{\varepsilon^{1/\beta}}\times 
\varepsilon^{1/\beta} b\leq \frac{1}{\varepsilon \beta}a^\beta + 
\varepsilon^{\beta'/\beta}\frac{b^{\beta'}}{\beta'},$$ with free parameter $\varepsilon>0$, in the last inequality.  A similar procedure on the second equation in \eqref{2000} gives, for any $\alpha > 1$, %and $\varepsilon>0$
%
%\be\label{4000}
$$\begin{array}{l}
\displaystyle\frac{\ud}{\ud t} \int_{\Omega} w^\alpha  \ud x
+ 4 \frac{\alpha-1}{\alpha}\int_{\Omega} |\nabla w^{\alpha/2}|^2  \ud x\\[.3cm]
\qquad\le \displaystyle \alpha \int_{\Omega}   c u w^{\alpha-1} \ud x + \alpha (\alpha-1) \int_{\Omega} w^{\alpha-1} \nabla w\cdot \nabla v  \ud x.
\end{array}$$
%\ee
The first integral of the right hand side is directly estimated by 
\[
\displaystyle\|c\|_{\infty}\left( \int_{\Omega} u^{\alpha}  \ud x+ \alpha \int_{\Omega} w^{\alpha}  \ud x\right),
\] 
as a consequence of H\"older and Young inequalities.
There are two ways to estimate the last integral, depending on the assumptions  
on $v$:
\begin{itemize}
\item[\small$\bullet$]
We 
recognize $ \alpha w^{\alpha-1} \nabla w=w^{\alpha/2}\times 2 \nabla w^{\alpha/2} $ and we
simply use the Cauchy-Schwarz inequality to obtain 
$$\begin{array}{l}\displaystyle
\Big| \alpha (\alpha-1) \int_{\Omega} w^{\alpha-1} \nabla w\cdot \nabla v  \ud x\Big|
\\[.3cm]
\leq
\displaystyle
 \|\nabla v\|_\infty \left(\frac{4(\alpha-1)}{\alpha} \int_\Omega|\nabla w^{\alpha/2}|^2\ud x \right)^{1/2}\left(\alpha(\alpha-1)\int_\Omega w^\alpha \ud x\right)^{1/2}
\end{array}$$
which leads to 
\be\label{4000-1}\begin{array}{l}
\displaystyle\frac{\ud}{\ud t} \int_{\Omega} w^\alpha  \ud x
+ 2 \frac{\alpha-1}{\alpha}\int_{\Omega} |\nabla w^{\alpha/2}|^2  \ud x\\[.3cm]
\qquad
\leq \displaystyle\|c\|_{\infty} \int_{\Omega} u^{\alpha}  \ud x +  \alpha \Big(\|c\|_{\infty}+\frac\alpha2
\|\nabla v\|^2_{\infty}\Big) \int_{\Omega} w^{\alpha}  \ud x
.
\end{array}\ee
\item[\small$\bullet$]
When \eqref{nv+} holds, we can integrate by parts so that 
$$
 \alpha (\alpha-1) \int_{\Omega} w^{\alpha-1} \nabla w\cdot \nabla v  \ud x=
 (\alpha-1) \int_\Omega \nabla w^\alpha \cdot \nabla v \leq
 (\alpha -1)\int_\Omega  w^\alpha \Delta v\ud x
$$
which yields 
\be\label{4000-2}\begin{array}{l}
\displaystyle\frac{\ud}{\ud t} \int_{\Omega} w^\alpha  \ud x
+ 4 \frac{\alpha-1}{\alpha}\int_{\Omega} |\nabla w^{\alpha/2}|^2  \ud x\\[.3cm]
\qquad\leq \displaystyle\|c\|_{\infty} \int_{\Omega} u^{\alpha}  \ud x +  \alpha\,\big(\|c\|_{\infty}+\|\Delta v\|_{\infty}\big) \int_{\Omega} w^{\alpha}  \ud x\,.
\end{array}\ee
\end{itemize}
Estimate \eqref{4000-1} is better in terms of required regularity for $v$, however, \eqref{4000-2} will play a role when $\alpha\rightarrow\infty$.  The latter will be used to show uniform propagation of the $L^{\infty}$-norm later on. 

Now, we use the fact that the space dimension is $n=2$ and we appeal to the following Gagliardo--Nirenberg--Sobolev inequality (see e.~g.~\cite[eq. (85) p.~195]{Brez}), which holds for any $\alpha\geq 1$,
 %%%%%%
 %%%%%
 \begin{equation}
\label{4500}\begin{array}{lll}
 \displaystyle\int_{\Omega} \xi^{\alpha+1}  \ud x &\le& C(\Omega,\alpha) \| \xi \|_{1}  \
\| \xi^{\alpha/2}\|_{H^1}^2 
\\
[.3cm]
&\le & C(\Omega,\alpha) \displaystyle\int_{\Omega} \xi  \ud x
\left(\int_\Omega  \xi^{\alpha}\ud x +\int_\Omega |\nabla (\xi^{\alpha/2})|^2\ud x\right).
\end{array}\end{equation}
Combined with $\int u(t,x) \ud x+ \int w(t,x)\ud x  = m_o$, it allows us to absorb the higher exponent of $u$ in the right side of \eqref{3000} %and \eqref{4000}  
by choosing $\varepsilon>0$ sufficiently small.  We can thus find positive constants $C(m_o)$ and $C'(m_o)$, depending on \eqref{HypH}, 
  and on the
 exponents $\alpha$ and $\gamma$, such that 
\begin{align}\label{4600}
\begin{split}
\frac{\ud}{\ud t} \int_{\Omega} u^\gamma  \ud x &+ C'(m_o)\int_{\Omega} u^{\gamma+1}  \ud x \\[.3cm]
&\leq C(m_o) \Big( \int_{\Omega} u^{\gamma}  \ud x + \int_{\Omega} w^{\gamma}  \ud x + \int_{\Omega}   w^{\gamma+1}  \ud x \Big)\,,\\[.3cm]
\frac{\ud}{\ud t} \int_{\Omega} w^\alpha  \ud x &+ C'(m_o)\int_{\Omega} w^{\alpha+1}  \ud x
\leq C(m_o) \Big( \int_{\Omega} u^{\alpha}  \ud x + \int_{\Omega} w^{\alpha}  \ud x\Big)\,.
\end{split}
\end{align}
 In order to control the right hand side  of inequalities \eqref{4600} exponents in the left hand side have to be bigger than exponents in the right;
  therefore, we are forced to choose $\gamma<\alpha<\gamma+1$.  We shall use Lebesgue's interpolation inequalities
\[
\begin{array}{ll}
\|w\|_{\gamma}\leq \|w\|^{1-\theta_1}_{1}\|w\|^{\theta_1}_{\alpha+1}\,,\quad &\theta_1 = \frac{(\gamma-1)(\alpha+1)}{\gamma\alpha}\in(0,1)\,,\\[.3cm]
\|w\|_{\gamma+1}\leq \|w\|^{1-\theta_2}_{1}\|w\|^{\theta_2}_{\alpha+1}\,,\quad &\theta_2 = \frac{\gamma(\alpha+1)}{\alpha(\gamma+1)}\in(0,1)\,.\\[.3cm]
\|u\|_{\alpha}\leq \|u\|^{1-\theta_3}_{1}\|u\|^{\theta_3}_{\gamma+1}\,,\quad &\theta_3 = \frac{(\gamma+1)(\alpha-1)}{\gamma\alpha}\in(0,1)\,.
\end{array}\]
Let us introduce
\begin{equation*}
U(t):=\int_{\Omega} u^{\gamma}(t)\ud x,\qquad
W(t):=\int_{\Omega} w^{\alpha}(t)\ud x.
\end{equation*}
>From \eqref{4600} we are thus led to
\begin{align*}
\frac{\ud}{\ud t} U &+ C'(m_o)\|u\|^{\gamma+1}_{\gamma+1} \leq C(m_o) \Big( U + \|w\|^{(\alpha+1)\theta_{1}}_{\alpha+1} + \|w\|^{(\alpha+1)\theta_{2}}_{\alpha+1} \Big)\,,\\[.3cm]
\frac{\ud}{\ud t} W &+ C'(m_o)\| w \|^{\alpha+1}_{\alpha+1} \leq C(m_o) \Big( W +  \|u\|^{(\gamma+1)\theta_{3}}_{\gamma+1} \Big)\,.
\end{align*}
For any $\varepsilon>0,  \beta\geq 1$, we can find $C(\varepsilon, \beta)>0$ such that $s\leq \varepsilon s^\beta + C(\varepsilon, \beta)
$ for any $s\geq 0.$
Thus, adding up these inequalities and setting $Z:=U+W$, it follows that there exists constants $C_{1}(m_0)$ and $C_{2}(m_o)$ such that
\begin{equation}\label{4650}
\frac{\ud}{\ud t}Z + C_{1}(m_o)\Big(\|u\|^{\gamma+1}_{\gamma+1}+\| w \|^{\alpha+1}_{\alpha+1}\Big)  \leq C_{2}(m_o)\big(Z+1\big)\, .
\end{equation}
%by using the elementary inequality $X^\theta\leq C_\theta(1+X)$ %which holds for any $X\geq 0$, $0\leq \theta\leq 1$.
Using Lebesgue's interpolation again
\begin{align*}
\|u\|_{\gamma}&\leq \|u\|^{1-\theta_4}_{1}\|u\|^{\theta_4}_{\gamma+1}\,,\qquad \theta_4=\frac{\gamma^{2}-1}{\gamma^{2}}\,,\\[.3cm]
\|w\|_{\alpha}&\leq \|w\|^{1-\theta_5}_{1}\|u\|^{\theta_5}_{\alpha+1}\,,\qquad \theta_5=\frac{\alpha^{2}-1}{\alpha^{2}}\,,
\end{align*}
we readily infer
\begin{align*}
\|u\|^{\gamma+1}_{\gamma+1}+\| w \|^{\alpha+1}_{\alpha+1}&\geq C(m_o)\Big(\|u\|^{(\gamma+1)/\theta_{4}}_{\gamma}+\|w\|^{(\alpha+1)/\theta_5}_{\alpha}\Big)
\\[.3cm]&\geq C(m_o)\Big(U^{\frac{\gamma}{\gamma-1}} + W^{\frac{\alpha}{\alpha-1}} \Big)\geq C(m_o)\,Z^{\frac{\alpha}{\alpha-1}}- C(m_o)\,,
\end{align*}
by using the fact that $s\mapsto \frac{s}{s-1}$ is non increasing.
Using $s\leq \varepsilon s^\beta + C(\varepsilon, \beta)$ again, and coming back to  \eqref{4650}, we arrive at
\begin{equation*}
\frac{\ud}{\ud t}Z + C'_{1}(m_o)\,Z^{\frac{\alpha}{\alpha-1}}\leq C'_{2}(m_o)\,,
\end{equation*}
with constants $C'_{1}(m_o)$ and $C'_{2}(m_0)$ depending on the initial mass $m_o$, \eqref{HypH} and the exponents $\alpha, \gamma$.  Therefore,  the   comparison principle in Corollary~\ref{A2} yields
\begin{equation*}
Z(t)\leq C(m_o)\Big(1+\frac{1}{t^{\alpha-1}} \Big)\,.
\end{equation*}
In other words, for any $1<\gamma<\alpha<\gamma+1<\infty$, we have
\begin{align*}
\|w(t)\|_{\alpha}&\leq C(m_o)\Big(1+\frac{1}{t^{\frac1{\alpha'}}} \Big)\,,\\[.3cm]
\|u(t)\|_{\gamma}&\leq C(m_o)\Big(1+\frac{1}{t^{\frac{\alpha-1}{\gamma}}} \Big)=C(m_o)\Big(1+\frac{1}{t^{(\frac{1}{\gamma'})^{+}}} \Big)\,.
\end{align*}
In the last equality we have used the fact that $\gamma<\alpha$ can be taken as close as desired.  
Furthermore, by invoking Lemma \ref{A1} the estimate on $t\mapsto Z(t)$ can be upgraded to $\sup_{t\geq0}Z(t)\leq C$  provided we add the dependence on the norms $\|w_o\|_{\alpha}$ and $\|u_o\|_{\gamma}$ in the constant.  This completes the proof of Proposition~\ref{T1000}.
\endproof
\subsection{From $L^{\gamma}$ to $L^{\infty}$ regularity.}
In this section we prove further gain of boundedness  for \eqref{2000}, showing that classical solutions are in fact bounded for any positive time.  We adopt  De Giorgi's energy method
\cite{DeG}, which has been successfully adapted to investigate the regularity of PDEs systems see e.~g.~\cite{CV, CV2, GV}.  Let us start the discussion and we shall state the theorem at the end of this section.\\[.3cm]

Consider a classical non negative solution $w$ of the equation
\begin{equation}\label{LpLinfe1}
\partial_{t}w - \Delta w + \nabla\cdot(w\nabla v) = f\,,\quad \text{in}\;\; [0,T]\times\Omega\,,
\end{equation}
with $v$ and $f$ given functions.
For the boundary condition we assume that $(\nabla w-w\nabla v)\cdot \nn=0$.
Define the level sets 
\begin{equation*}
w_{\lambda}= (w-\lambda)\,\textbf{1}_{\{w>\lambda\}}\,,\quad \lambda\geq0\,.
\end{equation*}
Multiply \eqref{LpLinfe1} by $w_{\lambda}$.  Owing to the boundary conditions, we obtain
\begin{equation*}
\displaystyle\frac{\ud}{\ud t}
\int_{\Omega}w_{\lambda}^{2}\, \ud x +2 \int_{\Omega}\big|\nabla w_{\lambda} \big|^{2}\, \ud x \leq 2\int_{\Omega} w\nabla v\cdot\nabla w_{\lambda}\, \ud x + 2\int_{\Omega}f_{+}\,w_{\lambda}\, \ud x.
\end{equation*}
   Young's inequality leads to
\begin{equation}\label{LpLinfe2}
\displaystyle\frac{\ud}{\ud t}\int_{\Omega}w_{\lambda}^{2}\, \ud x + %\displaystyle\frac{1}{2}
\int_{\Omega}\big|\nabla w_{\lambda} \big|^{2}\, \ud x \leq \int_{\Omega} \big| w\nabla v \big|^{2}\,\textbf{1}_{\{w>\lambda\}} \ud x + 2\int_{\Omega} f_{+}  \,w_{\lambda}\, \ud x.
\end{equation}

We split the reasoning  into two steps.
We start by  proving the $L^\infty$ bound  
on a given time interval $[t_\star,T]$.
Secondly, we shall extend the estimate to infinitely large time intervals.
Thus, let us consider $0<t_\star<T<\infty$.
Let $M>0$  and define the following sequence of levels and times
\begin{equation*}
\lambda_{k} = \big(1-1/2^{k}\big)M\,,\quad t_{k}= \big(1 - 1/2^{k+1} \big)t_{\star}\,,\quad k=0,1,2,... 
\end{equation*}
Define the following energy functional for the level sets
\begin{equation}\label{LpLinfe3}
W_{k}:= \sup_{t\in[t_{k},T]}\int_{\Omega}w^{2}_{k}\, \ud x +% \frac{1}{2}
\int^{T}_{t_{k}}\int_{\Omega} \big|\nabla w_{k} \big|^{2} \ud x\, \ud t\,,
\end{equation}
where we adopted the notation $w_{k}:=w_{\lambda_{k}}$.  With  $\lambda=\lambda_{k}$, let us integrate inequality \eqref{LpLinfe2} over the time interval $[s,t]$;  we obtain
\begin{align*}
\int_{\Omega}w^{2}_{k}(t,x) \ud x + %\frac{1}{2}
 \int^{t}_{s}\int_{\Omega}&\big|\nabla w_{k} \big|^{2} \, \ud x\ud t' \leq 
 \int_{\Omega}w^{2}_{k}(s,x) \ud x 
\\[.3cm]
&\hspace{-.3cm}+ %2
\int^{t}_{s}\int_{\Omega} \big| w\nabla v \big|^{2}\,\textbf{1}_{\{w>\lambda_{k}\}} \ud x\ud t' + 2\int^{t}_{s}\int_{\Omega} f_{+}  \,w_{k}\, \ud x\ud t'.
\end{align*}
We use this relation with $t_{k-1}\leq s\leq t_k\leq t\leq T$. It implies 
\[\begin{array}{l}
\displaystyle
W_{k}\leq  \int_{\Omega}w^{2}_{k}(s,x) \ud x \\[.3cm]
\qquad\quad+ \displaystyle\int^{T}_{t_{k-1}}\int_{\Omega} \big| w\nabla v \big|^{2}\,\textbf{1}_{\{w>\lambda_{k}\}} \ud x\ud t '+ 2\int^{T}_{t_{k-1}}\int_{\Omega} f_{+}  \,w_{k}\, \ud x\ud t'\,.
\end{array}\]
We take the mean over $s\in[t_{k-1},t_{k}]$, bearing in mind that $t_k-t_{k-1}=t_\star/2^{k+1}$. It yields
\begin{align}\label{LpLinfe4}
\begin{split}
W_{k}\leq & \frac{2^{k+1}}{t_{\star}}\int^{T}_{t_{k-1}}\int_{\Omega}w^{2}_{k} \ud x\ud s \\[.3cm]
&+ \int^{T}_{t_{k-1}}\int_{\Omega} \big| w\nabla v \big|^{2}\,\textbf{1}_{\{w>\lambda_{k}\}} \ud x\ud t '+ 2\int^{T}_{t_{k-1}}\int_{\Omega} f_{+}  \,w_{k}\, \ud x\ud t'\,.
\end{split}
\end{align}
We are going to make use of the  Gagliardo-Nirenberg interpolation inequality, see \cite[eq. (85), p. 195]{Brez},  
\begin{equation}\label{GNI}
\|w\|^{p}_{p} \leq C(\Omega,p,\alpha)\| w\|^{p\alpha}_{H^1}\|w\|^{(1-\alpha)p}_{2}\,,\quad 1=\Big(\frac{1}{2} - \frac{1}{n} \Big)\alpha\,p + \frac{1-\alpha}{2}\,p\,,
\end{equation}
which holds for any $\alpha\in[0,1]$ and $1\le p,q \le \infty$ (note that we perform the estimates without restricting the space dimension for the moment).  Thus, choosing $\alpha\,p=2$ it follows that
\begin{equation}\label{LpLinfe5}
\|w\|^{p}_{p} \leq C(\Omega,n)\| w\|^{2}_{H^1}\|w\|^{p-2}_{2}\,,\quad p=2\frac{n+2}{n}\,.
\end{equation}
Note that if $w_{k}>0$ then $w_{k-1}\geq 2^{-k}M$, and, as a consequence,
\begin{equation}
\label{alpha}
\textbf{1}_{\{w>\lambda_{k}\}} \leq \Big(\frac{2^{k}}{M}\,w_{k-1}\Big)^{a}\,,\quad \forall\,\,a\geq 0\,.
\end{equation}
Having this in mind one can play with the homogeneity of the right hand terms in the level set energy inequality \eqref{LpLinfe4}.  Indeed, 
with $a=4/n$, 
the first of them can be evaluated as follows
\begin{equation}
\begin{array}{l}\label{LpLinfe6}
\displaystyle\frac{2^{k+1}}{t_{\star}}\int^{T}_{t_{k-1}}\int_{\Omega}w^{2}_{k}\,\textbf{1}_{\{w>\lambda_{k}\}}\ud x\ud s \leq
\displaystyle \frac{2^{k+1}}{t_{\star}}\int^{T}_{t_{k-1}}\int_{\Omega}w^{2}_{k-1}\,\Big(\frac{2^{k}}{M}\,w_{k-1}\Big)^{\frac{4}{n}} \ud x\ud s\\[.3cm]
\qquad \displaystyle
\leq  2\frac{ 2^{\frac{4+n}{n}k} }{M^{\frac{4}{n}}t_{\star}}\int^{T}_{t_{k-1}}\int_{\Omega}w^{2\frac{n+2}{n}}_{k-1} \ud x\ud s \\[.3cm]
\qquad\displaystyle
 \leq
2C(\Omega,n)\ \frac{ 2^{\frac{4+n}{n}k} }{M^{\frac{4}{n}}t_{\star}}\int^{T}_{t_{k-1}}
\big(\|w_{k-1}\|_2^2+ \|\nabla w_{k-1}\|_2\big)\|w_{k-1}\|_2^{2\frac{n+2}{n}-2}
\ud s
\\[.3cm]
\qquad\displaystyle
 \leq 2C(\Omega,n)(1+T) \frac{ 2^{\frac{4+n}{n}k} }{M^{\frac{4}{n}}t_{\star}} W^{\frac{n+2}{n}}_{k-1},
\end{array}\end{equation}
by virtue of  \eqref{LpLinfe5} and the definition of $W_{k-1}$.  The last two terms 
in the right hand side of  \eqref{LpLinfe4} can be treated by using a similar procedure, together with 
 the application of H\"{o}lder's inequality. On the one hand, 
 bearing in mind that the $t_k$'s are all larger than $t_\star/2$,
 we get using \eqref{alpha} with $a =p$
\[\begin{array}{l}
\displaystyle\int^{T}_{t_{k-1}}\int_{\Omega} \big| w\nabla v \big|^{2} \,\textbf{1}_{\{w>\lambda_{k}\}} \ud x\ud t \leq \int^{T}_{t_{k-1}}\| w\nabla v \|^{2}_{2q'} \Big(\int_{\Omega}\textbf{1}_{\{w>\lambda_{k}\}} \ud x\Big)^{\frac{1}{q}}\ud t\\[.3cm]
\qquad\leq \displaystyle\frac{2^{\frac{p}{q}k}}{M^{\frac{p}{q}}}\sup_{t\geq \frac{t_{\star}}{2}}\| w\nabla v \|^{2}_{2q'}\int^{T}_{t_{k-1}}\Big(\int_{\Omega} w^{p}_{k-1} \ud x\Big)^{\frac{1}{q}}\ud t\\[.3cm]
\qquad
\leq \displaystyle C(\Omega,p,\alpha)^{\frac{1}{q}}\frac{2^{\frac{p}{q}k}}{M^{\frac{p}{q}}}
\left(\sup_{t\geq \frac{t_{\star}}{2}}\| w\nabla v \|^{2}_{2q'}
\right)
\int^{T}_{t_{k-1}}\| w_{k-1}\|^{\frac{p}{q}\alpha}_{H^1}\|w_{k-1}\|^{(1-\alpha)\frac{p}{q}}_{2}\ud t\,.
\end{array}\] 
We have used \eqref{GNI} in the last inequality.  We choose the parameters so that $\frac{p}{q}\alpha=2$ for some $0<\alpha<1$, which can be achieved as long as $1<q<\frac{n}{n-2}$ (more precisely,  going back to \eqref{GNI}, note that $\alpha=\frac{q}{1+2q/n}$
and $p=2(1+2q/n)$). It  follows that
\begin{equation}\label{LpLinfe7}\begin{array}{l}
\displaystyle
\int^{T}_{t_{k-1}}\int_{\Omega} \big| w\nabla v \big|^{2}\textbf{1}_{\{w>\lambda_{k}\}}\ud x\ud t 
\\[.3cm]
\qquad\leq 
(1+T)C(\Omega,p,\alpha)^{\frac{1}{q}}\frac{2^{\frac{2}{\alpha}k}}{M^{\frac{2}{\alpha}}}\left(\sup_{t\geq \frac{t_{\star}}{2}}\| w\nabla v \|^{2}_{2q'} \right)
W^{\frac{1}{\alpha}}_{k-1}\,.
\end{array}\end{equation}
On the other hand, since $w_{k} \leq w\,\textbf{1}_{\{w>\lambda_{k}\}}$, we have
\begin{align}\label{LpLinfe8}
\begin{split}
\int^{T}_{t_{k-1}}\int_{\Omega} f_{+}  \,w_{k}\, \ud x\ud t
& \leq \left(\sup_{t\geq \frac{t_{\star}}{2}}\|w\, f_+\|_{q'}\right)
\int^{T}_{t_{k-1}}\Big(\int_{\Omega} \textbf{1}_{\{w>\lambda_{k}\}}\, \ud x\Big)^{\frac{1}{q}}\ud t\\[.3cm]
&\hspace{-.5cm}\leq \frac{2^{\frac{p}{q}k}}{M^{\frac{p}{q}}}\left(\sup_{t\geq \frac{t_{\star}}{2}}\|w\,f_+\|_{q'}\right)
\int^{T}_{t_{k-1}}\Big(\int_{\Omega} w^{p}_{k-1}\, \ud x\Big)^{\frac{1}{q}}\ud t\\[.3cm]
&\hspace{-.8cm}\leq (1+T)C(\Omega,p,\alpha)^{\frac{1}{q}}\frac{2^{\frac{2}{\alpha}k}}{M^{\frac{2}{\alpha}}}
\left(\sup_{t\geq \frac{t_\star}{2}}\| f_{+} \|_{2q'}\| w \|_{2q'}\right)
 W^{\frac{1}{\alpha}}_{k-1},
\end{split}
\end{align}
still using the same relation between $p,q$ and $\alpha$.
We go back to \eqref{LpLinfe4},  with \eqref{LpLinfe6}, \eqref{LpLinfe7} and \eqref{LpLinfe8}: we arrive at
\begin{align}\label{LpLinf9}
\begin{split}
W_{k} \leq & (1+T)\Big[2C(\Omega,n)\frac{ 2^{\frac{4+n}{n}k} }{M^{\frac{4}{n}}t_{\star}} W^{\frac{n+2}{n}}_{k-1} \\[.3cm]
&\hspace{-.5cm}+ 2C(\Omega,p,\alpha)^{\frac{1}{q}}\frac{2^{\frac{2}{\alpha}k}}{M^{\frac{2}{\alpha}}}
\sup_{t\geq \frac{t_\star}{2}}
\Big(
\| \nabla v \|^{2}_{\infty} \| w \|_{2q'}^2+
\| f_{+} \|_{2q'} \| w \|_{2q'} \Big) W^{\frac{1}{\alpha}}_{k-1}\Big].
\end{split}
\end{align}
For the final step, we specialize to the case of space dimension $n=2$.
As a matter of fact, we observe that $1<q<\frac{n}{n-2}=\infty$.
We can then appeal to the estimates in Proposition~\ref{T1000} which imply  (particularizing to $f_{+}=u\,c$)
\begin{equation}\label{auxts}
\displaystyle\sup_{t\geq \frac{t_\star}{2}}
\Big(
\| \nabla v \|^{2}_{\infty} \| w \|_{2q'}^2+
\| f_{+} \|_{2q'} \| w \|_{2q'} \Big)
 \leq C(m_o)\Big(1+\frac{1}{t^{(\frac{q+1}{q})^{+}}_{\star}}\Big)\,.
\end{equation}
Therefore, \eqref{LpLinf9} becomes
\begin{equation}\label{DE}
W_{k} \leq C(1+T)\left(\frac{ 2^{3k} }{M^{2}t_{\star}} W^{2}_{k-1} + \Big(1+\frac{1}{t^{(\frac{q+1}{q})^{+}}_{\star}}\Big)\frac{2^{\frac{2(q+1)}{q}k}}{M^{\frac{2(q+1)}{q}}}W^{\frac{q+1}{q}}_{k-1}\right).
\end{equation}
We can take advantage of the fact that the powers $\frac{n+2}{n}=2$ and $\frac{q+1}{q}$ are strictly larger than 1 to conclude.  Indeed, a direct calculation shows that 
$W_{0}\,a^{k}$ is a super solution of the first order difference equation \eqref{DE} 
 with choices of $a\in(0,1)$ small, and then $M$ large (see e.~g.~\cite[Lemma~3.3 \&~3.4]{GV} for similar arguments) such that
\begin{equation*}
\begin{array}{l}
\displaystyle\max\Big\{ 2^{3 } a  ,\, 2^{\frac{2(q+1)}{q}} a^{ \frac{1}q } \Big\}<1\,,\\
[.3cm]
\displaystyle\max\Big\{\frac{C\,W_0 }{ a^{2 }\,M^{ 2 } t_{\star} } , \,\frac{C\,W^{ \frac{1}{q}  }_0 }{ a^{ \frac{q+1}{q} }\,M^{ 2\frac{q+1}{q} } }\Big(1+\frac{1}{t^{(\frac{q+1}{q})^+}_{\star}}\Big)   \Big\}\leq \frac{1}{2(1+T)}.\end{array}
\end{equation*}
By a comparison principle, we check that  $W_{0}\,a^{k}\geq W_{k}$ holds under these circumstances.  We conclude that 
\begin{equation*}
\lim_{k\rightarrow\infty} W_{k}=0\,.
\end{equation*}
%%%
%What follows needs to be corrected... but it would become really ugly!
%%
Let us suppose  $0<t_\star\ll 1$. 
Then, observe that  such a choice of $M$ explicitly takes the form
\begin{equation}\label{Main1}
M= \max\Big\{\sqrt{ \frac{ 2C(1+T)W_{0} }{a^2 t_{\star} }}
 ,\, \sqrt{ \displaystyle\frac{ (2C(1+T))^{\frac{q}{q+1}}W _{0}^{\frac{1}{q+1}}}{a}}\  \displaystyle\frac{2}{ t_{\star}^{(1/2)^+}  }\Big)\Big\}\,,
\end{equation}
where  the  constant $C$ depends on $m_o$, \eqref{HypH}, and of the exponent $q$.
We clearly have
\[
W_k\geq \displaystyle\frac{1}{T-t_\star}\displaystyle\int_{t_\star}^T\displaystyle\int_\Omega w^2(t,x)\ \mathbf 1_{\{w(t,x)\geq M(1-1/2^k)\}}\ud x\ud t.
\]
Letting $k$ run to 0, by virtue of Fatou's lemma we deduce that 
\[
 \displaystyle\frac{1}{T-t_\star}\displaystyle\int_{t_\star}^T\displaystyle\int_\Omega w^2(t,x)\ \mathbf 1_{\{w(t,x)\geq M\}}\ud x\ud t=0,
\]
which eventually implies 
\begin{equation}\label{Main2}
0\leq w(t,x)\leq M\qquad\text{ a.~e. $(t_\star,T)\times\Omega$.}
\end{equation}
This proves most of the following result.
\begin{proposition}\label{P1}
Let $(u,w)$ be a classical nonnegative solution of  \eqref{2000} with boundary conditions \eqref{2500}.  Then, the following $L^{\infty}$-estimate holds
\begin{align}\label{Bound}
\|u(t)\|_{\infty}+ \|w(t)\|_{\infty} &\leq  C(m_o)\Big(1+ \frac{1}{t^{1^+ }} \Big)\,, \quad t\in(0,T]\,,
\end{align}
where the constant $C$ depends additionally on \eqref{HypH}, but, it is independent of $T>0$.  Furthermore, 
when \eqref{nv+} holds, 
estimate \eqref{Bound} can be upgraded by adding the dependence on the $L^{\infty}$-norms of the initial data in the constant,
\begin{equation}\label{Bound1}
\max\big\{\|w(t)\|_{\infty}, \|u(t)\|_{\infty} \big\} \leq C\big(m_o,\|w_o\|_{\infty},\|u_o\|_{\infty}\big)\,, \quad t\geq0,
\end{equation}
where the constant is still independent of the time $T>0$.
\end{proposition} 
\proof
Proceeding as we did in \eqref{LpLinfe2}, we can establish the following inequality for $w$:
\begin{equation}\label{intW0}
\displaystyle\frac{\ud}{\ud t}\int_{\Omega}w^{2}\, \ud x + %\displaystyle\frac{1}{2}
\int_{\Omega}\big|\nabla w \big|^{2}\, \ud x \leq \int_{\Omega} \big| w\nabla v \big|^{2} \ud x + 2\int_{\Omega} f_{+}  w \ud x\,.
\end{equation}
We integrate over $[s,t]$ and we get, with $f_+=uc$,
\[\begin{array}{l}
\displaystyle\int_{\Omega}w^{2}(t,x)\, \ud x + %\displaystyle\frac{1}{2}
\int_s^t\int_{\Omega}\big|\nabla w(\sigma,x) \big|^{2}\, \ud x\ud\sigma 
\\[.3cm]
\qquad\leq
\displaystyle\int_{\Omega}w^{2}(s,x)\, \ud x 
+
\|\nabla v\|_\infty \int_s^t \int_{\Omega}  w^2(\sigma,x) \ud x \ud\sigma
\\[.3cm]
\qquad\qquad+
\|c\|_\infty\displaystyle\int_s^t 
 \int_{\Omega}  \big(u^2+w^2)(\sigma,x) \ud x\ud\sigma.
\end{array}\]
We use this relation with $s=t_\star/2\leq t\leq T$, which allows us to obtain the following estimate (recall the definition of $W_0$ in \eqref{LpLinfe3})
\begin{equation}\label{IC}\begin{array}{l}
W_{0} \leq 
%\sup_{ t\in[\frac{t_\star}{2},T] } \Big( \|w(t) \|^{2}_{2} +  \|u(t)\|^{2}_{2}\Big),
\|w\big(t_\star/2)\|^{2}_{2} 
\\[.3cm]
\qquad+
\displaystyle
 \Big(\|\nabla v\|_\infty+
\|c\|_\infty
\Big)\Big(T-\frac{t_{\star}}{2}\Big)\Big(\sup_{ t\in[\frac{t_\star}{2},T] } \|w(t) \|^{2}_{2} + \sup_{ t\in[\frac{t_\star}{2},T] } \|u(t)\|^{2}_{2} \Big)\,.
\end{array}\end{equation}
By Proposition \ref{T1000}, $W_0$ is thus dominated by $C(m_0)(1+T)(1+1/t_\star^{1^+})$.
%Pick $t>0$ and set $t_\star=t$, $T=2t$. 

Now, we set $T=1$, say.
Going back to \eqref{Main1} and \eqref{Main2}, we conclude that 
\begin{equation}\label{Linf}
\|w(t)\|_{\infty}\leq M\leq C(m_o)\ \frac{1}{ t_\star^{1^+} } 
\end{equation}
holds for short times, say $0< t_\star\leq t\leq T=1$.
It is clear that the same reasoning can be applied for any $T$, and it would provide a bound depending on $T$.
In order to justify that the $L^\infty$ estimate can be made uniform with respect to $T$, we thus proceed differently, by extending the estimate obtained on $[t_\star,1]$.
Indeed,  for any $t_1>0$, the shifted function $(t,x)\mapsto w_{t_1}(t,x)=w(t+t_1,x)$ is still a solution of  equation \eqref{LpLinfe1}, with data $w_{t_1}(0,x)=w(t_1,x)$ and the appropriate right hand side $f$.
In particular mass conservation implies $\int_\Omega (u(t_1,x)+w(t_1,x))\ud x=m_0$, so that the constant $C(m_o)$ does not change.
Therefore, we pick $0<t_1<T$ and 
we can repeat the same arguments for $w_{t_1}$,
which   leads to the same $L^\infty$ estimate \eqref{Linf} as for $w$  on the time interval $[t_\star,T]$, that is to say
\eqref{Linf} holds for $t\in [t_1+t_\star,t_1+T]$.
We have thus extended \eqref{Linf} over $[t_\star,t_1+T]$, see Figure~\ref{Fig:TimeShift}, and we can repeat this procedure.
  This completes the proof of \eqref{Bound} for $w$.

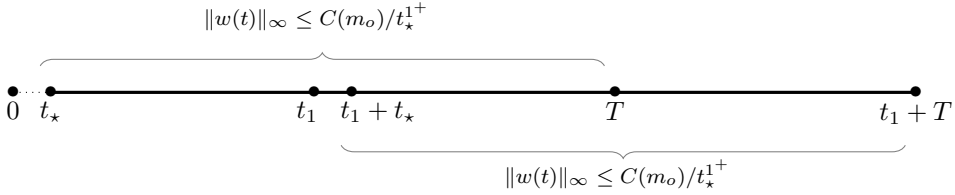
\begin{figure}[!ht]
\begin{center}
\begin{tikzpicture}
%\draw (0,2) node[midway,above]{$\Delta v e^{-D\Delta t}$};
%\draw[<->,>=latex] (-2,0) -- (2,0);
\draw [dotted] (-6,0) -- (-5.5,0);
\draw [very thick] (-5.5,0) -- (-2,0) -- (2,0) -- (6,0);
\draw (-6,0) node[below] {$0$} node[]{$\bullet$};
\draw (-5.5,0) node[below] {$t_\star$} node[]{$\bullet$};
\draw(-2,0) node[below] {$\!\! \! t_1$} node[]{$\bullet$};
\draw(-1.5,0) node[below] {$\qquad t_1+t_\star$} node[]{$\bullet$};
\draw(2,0) node[below] {$T$} node[]{$\bullet$};
\draw(6,0) node[below] {$t_1+T$} node[]{$\bullet$};
\draw [gray,decorate,decoration={brace,amplitude=5pt}, xshift=-4pt,yshift=10pt] 
(-5.5,0)  -- (2,0) 
   node [black,midway,above=10pt,xshift=-2pt]
{\footnotesize{$\|w(t)\|_\infty\leq C(m_o)/t_\star^{1^+}$}};
\draw [gray,decorate,decoration={brace,amplitude=5pt,mirror}, xshift=-4pt,yshift=-20pt] 
(-1.5,0)  -- (6,0) 
   node [black,midway,below=2pt,xshift=-2pt]
{\footnotesize{$\|w(t)\|_\infty\leq C(m_o)/t_\star^{1^+}$}};
\end{tikzpicture}
\caption{Extension of the local result.}
\label{Fig:TimeShift}\end{center}
\end{figure}

Next, using elliptic regularity for the pheromone equation, see \cite{ADN}, and Sobolev's embedding (still with space dimension $n=2$), it follows that
\begin{equation*}
\|\nabla p(t)\|_{\infty}\leq C\|p(t)\|_{W^{2,2^+}}
 \leq C(\Omega)\,\|w(t)\|_{2^+}\leq C(m_o)\,\Big(1 + \frac{1}{t^{(1/2)^+}}\Big)\,.
\end{equation*}
Equipped with this estimate we repeat for $u$ the arguments 
used for estimating $w$,
just changing $c$ to $N$ and replacing the uniform estimate on $\nabla v$ by 
this time-dependent estimate for $\nabla p$.
 This finishes the proof of estimate \eqref{Bound}.
  As a matter of fact, note that \eqref{Linf} applies in particular to $t=t_\star$:
 $\|w(t_\star)\|_\infty\leq C(m_o)/t_\star^{1^+}$ holds for any $0<t_\star\leq 1$.
\\

For proving  the uniform bound \eqref{Bound1} let us set
\begin{equation}\label{eqA}
A:=\max\Big\{\sup_{s\geq0}\|u(s)\|_{2^{+}},\,\sup_{s\geq0}\|w(s)\|_{2^{+}}\Big\}<\infty.
\end{equation}
We slightly modify the previous analysis.
In particular, we change the definition of the $t_k$'s
by setting $t_k=t_\star(1-1/2^k)$, so that it now starts from $t_0=0$.
Going back to \eqref{intW0}, we get
$W_0\leq \|w_0\|^{2}_{2}+C\,T$, with $C$ depending on \eqref{HypH} and the constant $A$ in \eqref{eqA}.
Furthermore, 
 the additional information in \eqref{eqA} allows us to control uniformly with respect to time the $L^{2q'}$ norms
(bearing in mind that making $q$ large means $q'$ close to 1) so that \eqref{auxts}
becomes
$$\displaystyle\sup_{t\geq 0}
\Big(
\| \nabla v \|^{2}_{\infty} \| w \|_{2q'}^2+
\| f_{+} \|_{2q'} \| w \|_{2q'} \Big)
 \leq C.
$$
Accordingly,
 inequality \eqref{DE} is changed into
\begin{equation*}
W_{k} \leq C_0\frac{ 2^{3k} }{M^{2}t_\star} W^{2}_{k-1} + C_{1}\frac{2^{\frac{2(q+1)}{q}k}}{M^{\frac{2(q+1)}{q}}}W^{\frac{q+1}{q}}_{k-1},
\end{equation*}
where $C_0$ and $C_1$ depend on \eqref{HypH}, but also on $T$ and  $A$, while they do not depend on $k$ and $t_\star$.
Reproducing  the same argumentation as above permits us to establish that
\begin{equation}\label{Bound+1}
\|w(t)\|_{\infty} \leq  C(m_o,A)\Big(1+ \frac{1}{\sqrt{t}} \Big)\,,
\end{equation}
holds for  $0<t\leq 1$.
Furthermore, elliptic regularity \cite{ADN} yields 
\begin{equation}
\label{SobEmb}
\|\nabla p(t)\|_{\infty}\leq C(\Omega)\|w(t)\|_{2^{+}}\leq A 
\end{equation}
and, as a consequence, we get a similar estimate for $u$:
\begin{equation}\label{Bound+2}
\|u(t)\|_{\infty} \leq  C(m_o,A)\Big(1+ \frac{1}{\sqrt{t}} \Big)
\end{equation}
holds for  $0<t\leq 1$, too.
Let us set $$U(t):=\int_\Omega u^{\gamma}(t,x)\ud x,\qquad
W(t):=\int_\Omega w^{\gamma}(t,x)\ud x.$$ 
Since \eqref{nv+} holds, we can make use of 
 inequalities \eqref{3000} and \eqref{4000-2} (with $\alpha=\gamma$), and  we get
\begin{align}\label{UNIBOUND}
\begin{split}
\frac{\ud U}{\ud t} &\leq \|N\|_\infty W + \gamma\big( \|N\|_\infty  + \|w(t)\|_{\infty}\big)U\\[.3cm]
&\leq \|N\|_\infty W + \gamma B\Big( 1 + \displaystyle\frac{1}{\sqrt{t}}\Big)U\,, \\[.3cm]
\frac{\ud W}{\ud t} &\leq \|c\|_\infty U + \gamma \big( \|c\|_\infty + \|\Delta v\|_{\infty}\big)W\,,
\end{split}
\end{align}
for $t\in (0,1)$, with $B>0$ independent of $\gamma>0$.  Adding the two inequalities in \eqref{UNIBOUND},
we deduce that  $Z:=U+W$ satisfies
\begin{equation*}
\frac{\ud Z}{\ud t}\leq C\gamma\Big( 1 + \displaystyle\frac{1}{\sqrt{t}}\Big)Z,
\end{equation*}
for $t\in (0,1)$, and a certain constant $C>0$ which does not depend on $\gamma$.
It leads to
\begin{equation*}
Z(t) \leq Z_0 \exp\Big(C\gamma\int^{t}_{0}\Big( 1 + \displaystyle\frac{1}{\sqrt{s}}\Big)\ud s \Big)
\leq Z_0e^{3C\gamma},
\end{equation*}
which eventually implies the estimate
\begin{equation*}
\max\{\|u(t)\|_{\gamma},\|w(t)\|_{\gamma}\}\leq \max\{\|u_o\|_{\gamma},\|w_o\|_{\gamma}\}e^{3C}
\end{equation*}
which thus  holds a.e.~$t\in (0,1)$.
Sending $\gamma\rightarrow\infty$ proves estimate \eqref{Bound1} for short times.

Going back to the second part of the statement in Proposition~\ref{T1000}, we realize that \eqref{eqA} 
is ensured when the initial 
data belong to $L^1\cap L^\infty(\Omega)$.
This observation ends the proof of the $L^\infty$ estimate for  $0\leq t\leq 1$.
Combined to \eqref{Bound}, it  provides a uniform estimate on the solution for any $t\geq 0$.
\endproof
%
%\subsection{Improved regularity and classical solutions}
\subsection{Construction of classical solutions}

We have now the tools to prove the global well-posedness of the  system  \eqref{2000}.  
Firstly, we establish the existence--uniqueness of solutions locally in time, and, secondly, we extend the result, as a consequence of the obtained estimates.
We start by working with smooth and bounded initial data, say $(u_o,w_o)\in C^\infty_c(\Omega)$.

\begin{theorem}[Classical solutions]\label{T500}
Let $u_o,\,w_o \in C^{\infty}_c(\Omega)$ be nonnegative initial data. Then, for every $T>0$ the system~\eqref{2000} with boundary conditions \eqref{2500} admits a unique nonnegative classical solution $(u,w,p) \in \mathcal{Y}$ which satisfies the estimates of Propositions \ref{T1000} and \ref{P1}.
% and \ref{FDE}. 
\end{theorem}

\proof
As explained above the proof splits into two steps.

\noindent
\textit{Step 1: Local existence.}
Let us introduce the 
  convex set
\begin{equation*}
\mathcal{Y} =\Big\{
\xi\in 
C\big([0,T];L^{2}(\Omega)\big) \cap L^{2}\big(0,T; H^{1}(\Omega) \big),\,
0\leq \xi(t,x)\leq 2\|w_o\|_\infty\Big\}.
\end{equation*}
This space is  endowed  with the norm
\[
\|u\|_{\mathcal{Y}}=\displaystyle\sup_{0\leq t\leq T} \|u(t,\cdot)\|_{L^2}.
\]
Let $\varphi \in \mathcal{Y}$ be such that $\varphi(0)=w_o$.  To such $\varphi$ we associate  the solution $p(\varphi)$ of the pheromone equation
\begin{equation}
\label{2510}
\aligned
-\Delta p = \varphi - \delta p,
\endaligned
\end{equation}
and, then, we associate the solution $u(\varphi) = u(\varphi, p(\varphi))$ of the equation
\begin{equation}
\label{2511}
\aligned
\partial_t u -  \Delta u + \dive\big( u\, \nabla p(\varphi) \big) = - uc + \varphi N\,,\quad u(0)=u_o\,.
\endaligned
\end{equation}
Finally, let $w(\varphi) = w( \varphi, u(\varphi) )$ be the solution of 
\begin{equation}
\label{2512}
\aligned
\partial_t w -  \Delta w + \dive\big( w\,  {\nabla v}   \big) =  u c - w N\,,\quad w(0)=w_o\,.
\endaligned
\end{equation}
Equations are complemented with the zero-flux boundary conditions \eqref{2500}.  
Note that \eqref{2510}, \eqref{2511} and \eqref{2512} are linear equations
and existence of solutions is provided by the standard theory, see e.~g.~\cite[Theorem X.9]{Brez}.
In particular both $u(\varphi)$ and $w(\varphi)$ belong to  $C\big([0,T];L^{2}(\Omega)\big) \cap L^{2}\big(0,T; H^{1}(\Omega))$.
We aim at finding $T$ small enough, so that solutions of \eqref{2000} can be obtained as fixed points of the mapping
\begin{equation*}
\Phi :  \varphi  \longmapsto w(\varphi).
\end{equation*}
We shall  prove that $\Phi$ is a contraction in $\mathcal Y$,
for a sufficiently small time $T>0$ depending only on \eqref{HypH} and the initial data $(u_o,w_o)$.  
This is a consequence of the maximum principle for the linear equation
\begin{equation*}
\partial_ t \psi-\Delta\psi +\nabla\cdot (B\psi)+ b\psi =f,
\end{equation*}
complemented with the boundary condition
\[\nabla \psi\cdot \nn-\psi B\cdot \nn=0.\]
We assume (related to \eqref{nv+})
\begin{equation*}
b\geq 0, \qquad f\geq 0,
\qquad B\cdot \nn \leq 0\,.
\end{equation*}
Then, the solution $\psi$ associated to $\psi(0,x)=\psi_o(x)\geq 0$ is non negative.
This result can be obtained by considering the identity
\[\begin{array}{l}
\displaystyle\frac{\ud}{\ud t}\int_\Omega G(\psi)\ud x+\int _\Omega
G''(\psi)|\nabla \psi|^2\ud x
\\[.3cm]
\qquad=\displaystyle\int _\Omega\psi G''(\psi)B\cdot\nabla \psi\ud x-\int _\Omega
b\psi G'(\psi)\ud x+\int _\Omega
fG'(\psi)\ud x.
\end{array}\]
We work with the convex function $G(z)=\frac{[z]_-^2}{2}$, $[z]_-=\min(z,0)\leq 0$.  It follows that $G'(z)=[z]_-\leq 0$, and $zG'(z)\geq 0$, $zG''(z)=G'(z)$ so that $|\psi G''(\psi)\nabla \psi|
=\sqrt{2G(\psi)}\times \sqrt{|G''(\psi)|} \nabla \psi|
$.
Then Cauchy-Schwarz and Young inequality yield
\begin{equation*}
\displaystyle\frac{\ud}{\ud t}\int _\Omega
G(\psi)\ud x\leq \|B\|^2_\infty \int _\Omega
G(\psi)\ud x.
\end{equation*}
We conclude as a consequence of the Gr\"onwall lemma that $\psi(t,x)\geq 0$.
Let us set $$\gamma(t)=e^{t(\|\nabla\cdot B\|_\infty + \|b\|_\infty)}(t\|f\|_\infty +\|\psi_o\|_\infty).$$
Observe that $\lim_{t\to 0}\gamma (t)=\|\psi_o\|_\infty$.
Applying the previous maximum principle to $\tilde\psi(t,x):=\gamma(t)-\psi(t,x)$, we conclude that $\psi(t,x)\leq \gamma(t)$.  Indeed, note that $ \tilde\psi(t,x)$ satisfies the non homogeneous Robin's condition $(\nabla\cdot \tilde\psi -B\tilde\psi)\cdot \nn =-\gamma B\cdot \nn$ with a non negative source $-\gamma B\cdot \nn\geq 0$ so that 
the computation now involves the boundary term $-\int_{\partial\Omega} B\cdot\nn \ \gamma G'(\gamma - \psi) \ud \sigma$ which contributes negatively owing to the orientation of $B$ towards the interior of $\Omega$.

Since $0\leq \varphi(t,x)\leq 2\|w_o\|_\infty$ holds on $[0,T]\times\Omega$, 
we have that $p$ and $\Delta p$ are bounded in $L^\infty((0,T)\times \Omega)$, and $\nabla p$ as well. 
This allows us to use the above maximum principle with the solution of \eqref{2511}, and then also to the solution of \eqref{2512}, to justify that $0\leq w(t,x)\leq 2\|w_o\|_\infty$ holds a.~e.~$(t,x)\in [0,T]\times\Omega$, provided $T$ is small enough.
We thus have $\Phi(\mathcal{Y})\subset \mathcal{Y}$.

Let us show that $\Phi$ is a contraction in $\mathcal{Y}$, possibly at the price of making  $T>0$ smaller. 
Let $\varphi_1,\, \varphi_2 \in \mathcal{Y}$ and let $(u_i,\, w_i,\,p_i)$, with $i=1,2$, be the corresponding solutions of respectively, \eqref{2510}, \eqref{2511}, and \eqref{2512}. Define $\overline w := w_1 - w_2=\Phi(\varphi_1)-\Phi(\varphi_2)$ and so forth.  We  compute
$$\begin{array}{l}
\displaystyle\frac12\frac{\ud}{\ud t}\int _\Omega
|\bar w|^2\ud x +\int _\Omega
|\nabla\bar w|^2\ud x
+\int _\Omega
N|\bar w|^2\ud x
\\[.3cm]
\qquad
=\displaystyle
\int _\Omega
c \bar u\bar w\ud x
+\int _\Omega
\bar w\nabla v\cdot\nabla \bar w\ud x
\\
[.3cm]
\qquad
\leq \displaystyle\frac12(\|\nabla v\|^2_\infty+ \|c\|^2_\infty) \int _\Omega
|\bar w|^2\ud x
+\displaystyle\frac12\int _\Omega
|\nabla\bar w|^2\ud x
+\displaystyle\frac12\int _\Omega
|\bar u|^2\ud x
 .
\end{array}$$
Gr\"onwall's lemma thus yields
\begin{equation}
\label{2512.1}
\int _\Omega
|\bar w|^2\ud x\leq e^{K_1t}\int_0^t\int _\Omega
|\bar u|^2\ud x\ud s
\end{equation}
with $K_1=\|\nabla v\|^2_\infty+ \|c\|^2_\infty$.
We proceed similarly with \eqref{2511}
and we get 
$$
\begin{array}{l}
\displaystyle\frac12\frac{\ud}{\ud t}\int _\Omega
|\bar u|^2\ud x +\int _\Omega
|\nabla\bar u|^2\ud x
+\int _\Omega
c|\bar u|^2\ud x
\\
[.3cm]
\qquad=\displaystyle\int _\Omega
N \bar u\bar \varphi\ud x
+\int _\Omega
\bar u \nabla p_1 \cdot\nabla \bar u\ud x
+
\int _\Omega
u_2\nabla \bar p\cdot\nabla \bar u\ud x
\\
[.3cm]
\qquad\leq
\displaystyle\frac12\|N\|^2_\infty \int _\Omega
|\bar u|^2\ud x + \displaystyle\frac12  \int _\Omega
|\bar \varphi|^2\ud x
+
\displaystyle\frac12\|w_1\|_\infty \int _\Omega
|\bar u|^2\ud x
\\
[.3cm]
\qquad\qquad+
\displaystyle\frac12
\int _\Omega
|\nabla\bar u|^2\ud x
+
\displaystyle\frac12 \|u_2\|^2_\infty \displaystyle\int _\Omega
|\nabla \bar p|^2 \ud x,
\end{array}
$$
where we have used the equation $\delta p_1-\Delta p_1=w_1$.
Since $0\leq \varphi_i(t,x)\leq 2\|w_o\|_\infty$ holds on $[0,T]\times\Omega$, 
we deduce that $p_2$ and $\Delta p_2$ are bounded in $L^\infty((0,T)\times \Omega)$,
and thus, the maximum principle applies for \eqref{2511} and we can assume  
that $\|u_2\|_\infty\leq  2\|u_o\|_\infty$ holds on $[0,T]$.
Furthermore, we have $\|\nabla \bar p\|^2_2\leq C\|\bar \varphi\|^2_2$, with $C$ depending on $\delta \geq 0$.
It follows that 
$$
\frac{\ud}{\ud t}\int _\Omega
|\bar u|^2\ud x
\leq 
(\|N\|^2_\infty +2\|w_o\|_\infty) \int _\Omega
|\bar u|^2\ud x +  
(1+4C\|u_o\|^2_\infty)  \int _\Omega
|\bar \varphi|^2\ud x
$$
holds on $[0,T]$.
Hence, we get
$$
\int _\Omega
|\bar u|^2\ud x\leq (1+4C\|u_o\|^2_\infty)
e^{K_2t}  \int_0^t \int _\Omega
|\bar \varphi|^2\ud x\ud s,
$$
with $K_2=\|N\|^2_\infty +2\|w_o\|_\infty$.
Finally (going back to \eqref{2512.1}), we find $K,M>0$, depending on \eqref{HypH} and the $L^\infty$ norms 
of the data $(u_0,w_o)$,
such that 
$$
\int _\Omega
|\bar w(t,x)|^2\ud x\leq Me^{Kt}\frac{t^2}{2}\sup_{0\leq s\leq t} \int _\Omega
|\bar\varphi(s,x)|^2\ud x,
$$
holds on $[0,T]$. This justifies that $\Phi$ is a contraction on $\mathcal Y$, for 
a small enough time interval.
The fixed point defines a solution of \eqref{2000}.
\\

\noindent
\textit{Step 2: Global solutions.} 
We remind the reader that we are working with data $(u_o,w_o)\in C^\infty_c(\Omega)$.
The obtained solution is bounded on $(0,T)\times \Omega$.
Therefore, the classical theory of parabolic equations tells us that $\nabla u$ and $\nabla p$ are bounded
on $[t_\star,T]\times \Omega'$, for any $0<t_\star<T$ and $\Omega'$ strictly included in $ \Omega$;
see for instance \cite[Th. VII.6.1]{LUS}.
We can then perform a bootstrap argument to show that the solution is $C^\infty([t_\star,T]\times \Omega')$, if the coefficients $N, c, v$ are assumed to be $C^\infty$.
Moreover, by virtue of $(u_o,w_o)\in C^\infty_c(\Omega),$ the data satisfy compatibility conditions with the boundary conditions, so that we have actually constructed classical solutions, see
\cite[Th. 4.3]{LiMa} or \cite[Th. V.7.4]{LUS}.
In particular, the a priori estimates discussed above apply.
Proposition~\ref{P1} provides a uniform $L^\infty$ estimate on the solution, which depends only on \eqref{HypH} and $\|u_o\|_\infty, \|w_o\|_\infty$.
Therefore, we can reproduce the fixed point argument and extend the solution over $[0,\infty)$. \endproof

We finish this section studying the stability of classical solutions in the space
\begin{equation*}
\mathcal{S}:=\Big\{(u,w)\in L^{\infty}\big(0,T; L^{1} \cap L^{2^{+}}(\Omega)\big)\Big\}\,,\quad T>0\,.
\end{equation*}
\begin{proposition}\label{unique}
Nonnegative classical solutions  of \eqref{2000} with boundary conditions \eqref{2500} are stable in $\mathcal{S}$.  More precisely, given two initial data $(u_{i,o},w_{i,o})\in L^{1} \cap L^{2^+}(\Omega)$, with $i=1,2$, then solutions $(u_{i},w_{i})$ satisfy the estimate
\begin{align*}
\|u_{1}(t)-u_{2}(t)\|_{2}&+\|w_{1}(t)-w_{2}(t)\|_{2}\\[.3cm]
&\leq \Big(\|u_{1,o}-u_{2,o}\|_{2}+\|w_{1,o}-w_{2,o}\|_{2}\Big)e^{C(m_o)t}\,,\quad t\geq0\,.
\end{align*}
The constant $C(m_o)$  depends on \eqref{HypH} and the $\|\cdot\|_{2^{+}}$-norm of the data. 
\end{proposition}
\begin{proof}
Let $\overline{u} = u_{1} - u_{2}$ and similarly for the other differences.  The system for the differences reads
\begin{align}\label{Ue1}
\begin{split}
\partial_{t}\overline u - \Delta \overline u + \nabla\cdot(\overline u \nabla p_1) + \nabla\cdot(u_{2}\nabla\overline p) &= -c\overline u + N \overline w\,,\\[.3cm]
\partial_{t}\overline w - \Delta \overline w + \nabla\cdot(\overline w \nabla v) &= c\overline u  - N\overline w\,,\\[.3cm]
-\Delta \overline p &= \overline w - \delta\overline p\,.
\end{split}
\end{align}
Multiplying the second equation in \eqref{Ue1} by $\overline w$, integrating by parts and using Cauchy-Schwarz, we are led to 
\begin{align*}
\frac{1}{2}\frac{\ud}{\ud t}\|\overline w\|^{2}_{2} + &\|\nabla \overline w\|^{2}_{2} \leq \|\overline w\|_{2}\Big(\|\nabla v\|_{\infty}\|\nabla \overline w\|_{2} + \|c\|_{\infty}\|\overline u\|_{2}\Big)\,,
\end{align*} 
which yields
\begin{equation}\label{UNIw}
\frac{\ud}{\ud t}\|\overline w\|^{2}_{2} + \|\nabla \overline w\|^{2}_{2} \leq C\Big(\|\overline w\|^{2}_{2} + \|\overline u\|^{2}_{2} \Big)\,.
\end{equation}
for $C=\|\nabla v\|^{2}_{\infty} + \|c\|^{2}_{\infty}$.  Next, using the equation for $u$ we get
\begin{equation*}
\frac{\ud}{\ud t}\|\overline u\|^{2}_{2} + \|\nabla\overline u \|^{2}_{2}\leq 2\Big(\|\overline u \nabla p_{1} \|^{2}_{2} + \|u_{2} \nabla \overline p \|^{2}_{2}\Big) + \| \overline u \|^{2}_{2} + \|\overline w N\|^{2}_{2}\,,
\end{equation*}
The right side terms above are easily controlled using  \eqref{eqA}, \eqref{Bound+1}, \eqref{SobEmb}, \eqref{Bound+2}
\begin{align*}
\|\overline u(t) \nabla p_{1}(t) \|^{2}&\leq \|\nabla p_{1}(t)\|^{2}_{\infty}\|\overline u(t)\|^{2}_{2}\leq C\|\overline u(t)\|^{2}_{2}\,,\\[.3cm]
\|u_{2}(t) \nabla \overline p(t) \|^{2}_{2}&\leq \|u_{2}(t)\|^{2}_{2q'}\|\nabla\overline p(t)\|^{2}_{2q} \leq \Big(\|u_{2}(t)\|^{1/q'}_{1}\|u_{2}(t)\|^{1+1/q}_{\infty}\Big)\|\overline p(t)\|^{2}_{H^{2}}\\[.3cm]
&\hspace{1cm}\leq C\Big(1+\frac{1}{t^{\frac{(q+1)}{2q}}} \Big)\|\overline w(t)\|^{2}_{2}\,,\quad q\in[1,\infty)\,.
\end{align*}
The constant $C$ now depends on $m_o$ and on the $L^{2^+}$ norm of the data.  Taking $q$ arbitrarily large, we obtain
\begin{equation}\label{UNIu}
\frac{\ud}{\ud t}\|\overline u\|^{2}_{2} + \frac{1}{2}\|\nabla\overline u \|^{2}_{2}\leq C\Big(1+\frac{1}{t^{(1/2)^{+}}} \Big)\Big(\|\overline u\|^{2}_{2} + \|\overline w\|^{2}_{2}\Big)\,.
\end{equation}
Now define $Z(t):=\|\overline u\|^{2}_{2} + \|\overline w\|^{2}_{2}$.  Thus, adding up estimates \eqref{UNIw} and \eqref{UNIu} it follows that
\begin{equation*}
\frac{\ud Z(t)}{\ud t}\leq C\Big(1+\frac{1}{t^{(1/2)^{+}}} \Big) Z(t)\,.
\end{equation*}
Integration of this ODE leads to the result.
\end{proof}
\subsection{Global well-posedness of weak solutions}
We are going to prove the global existence of weak solutions, in the sense of Definition~\ref{dws}.
\begin{theorem}[Global well-posedness]\label{TFINAL2}
Let arbitrary $T>0$, $\delta>0$ be fixed and assume nonnegative initial data $(u_o,w_o)\in L^{1}\cap L^{2^{+}}(\Omega)$.  Then, there exists a unique nonnegative weak solution for the system \eqref{2000}.  Such solution satisfies the estimates of Propositions \ref{T1000} and \ref{P1}. 
%Moreover, it is a classical solution for any $t>0$. 

In the case $\delta=0$ uniqueness continues  holding up to a constant in the pheromone $p$ distribution.
\end{theorem}
\proof
Take a sequence of non negative  initial data $\big(u^{k}_{o},w^{k}_{o}\big) \in C^{\infty}_c(\Omega)$ such that
\begin{equation*}
\big(u^{k}_{o},w^{k}_{o}\big) \rightarrow \big(u_{o},w_{o}\big)\quad\text{strongly in}\;\; L^{1}\cap L^{2^{+}}(\Omega)\,.
\end{equation*}
Using Theorem \ref{T500} on global well-posedness of classical solutions, we have a sequence $(u^{k},w^{k},p^{k})\in C\big(0,T; L^{2}(\Omega) \big)$ of solutions to the  system  \eqref{2000}.  It is not difficult to check that, in fact, such a sequence is uniformly bounded in $L^{2}\big(0,T; H^{1}(\Omega) \big)$ with time derivatives uniformly bounded in $L^{2}\big(0,T; H^{1}(\Omega)^{\star}\big)$.  Let
\begin{equation*}
Z_{kl}(t)=\|u^{k}(t)-u^{l}(t)\|^{2}_{2}+\|w^{k}(t)-w^{l}(t)\|^{2}_{2}\,,\quad k,\,l\geq 1\,,
\end{equation*}
be the Cauchy differences.  Note that Proposition \ref{T1000} implies that $(u^{k},w^{k})\in L^{\infty}\big(0,T; L^{2^{+}}(\Omega)\big)$ \textit{uniformly} with respect to  $k\geq1$, hence, the stability result of Proposition \ref{unique} implies that
\begin{equation*}
Z_{kl}(t)\leq Z_{kl}(0)e^{Ct}\,,
\end{equation*}
with constant $C$ independent of the indices $k,\,l$.  In this way we conclude that both $u^k$ and $w^k$ are Cauchy in $L^{\infty}\big(0,T;L^{2}(\Omega)\big)$.  Thus,
the following convergence properties hold:
\begin{align*}
\big(u^{k},w^{k}\big) \rightarrow \big(u,w\big)&\quad\text{strongly in}\;\; L^{\infty}\big(0,T;L^{2}(\Omega)\big)\,,\\[.3cm]
\big(u^{k},w^{k}\big) \rightarrow \big(u,w\big)&\quad\text{weakly in}\;\; L^{2}\big(0,T;H^{1}(\Omega)\big)\,\\[.3cm]
\big(\partial_t u^{k},\partial_t w^{k}\big) \rightarrow \big(\partial_t u ,\partial_t w \big)&\quad\text{weakly in}\;\; L^{2}\big(0,T;H^{1}(\Omega)^{\star}\big)\,.
\end{align*}
With standard estimates for the elliptic
problems, we deduce that $\nabla p^k$ is bounded in
$L^{\infty}(0,T;L^{2^+}(\Omega))$, and $\nabla p_k$ converges to
$\nabla p$ strongly in $L^\infty(0,T;L^{2}(\Omega))$. Accordingly the
product $u_k\nabla p_k$ converges to $u\nabla p$ strongly  in
$L^1((0,T)\times\Omega)$. 
As a consequence, $0\leq (u,w)\in L^{\infty}\big(0,T;L^{2}(\Omega)\big)$ is a weak solution of  \eqref{2000}.  Note that the initial condition $(u(0), w(0))=(u_o,w_o)$ is satisfied by continuity at $t=0$ which follows from the estimate on the time derivatives.

Finally, using approximation by classical solutions again one proves that the stability result of Proposition \ref{unique} is valid for weak solutions as well.  Uniqueness follows from here.
\endproof

\begin{remark}
It is likely that the $L^{2^+}$ integrability of the data is not optimal for the theory of existence of solutions.
However, it simplifies the analysis in two directions. On the one hand, it makes the definition, and the stability,  of the product $u\nabla p$
meaningful. On the other hand,  the underlying estimate is also useful in the proof of Proposition~\ref{unique} which implies the uniqueness of weak solutions. 
\end{remark}

%
%
%%%%%%%%%%%%%%%%%%%%%%%%%%%%%%%%%%%%
%%%%%%%%%%%%%%%%%%%%%%%%%%%%%%%%%%%%
%%%%%%%%%%%%%%%%%%%%%%%%%%%%%%%%%%%%
%
%

\section{Analysis of the  model \eqref{1000}}\label{4}
In this section we perform the analysis for the system \eqref{1000} that includes regularity estimates and global well-posedness in the case where the problem is set on the entire space $\Omega=\mathbb{R}^{2}$.  A short comment for the case $\Omega=\mathbb R^n$ with $n\geq3$ is included in the Remark~\ref{Ren=3}.  Details are given only in the arguments for obtaining \textit{a priori} estimates since the ideas to make such estimates rigorous were previously presented (and are quite classical).  We remind the reader that the main hypotheses imposed on the parameters are gathered in \eqref{HypH}.  Eq. \eqref{1000}  involves an evolution equation for the food concentration $c$; therefore, in addition we will assume   

\be
\label{HSPD}\tag{\bf {H+}}
c_o\in L^{\infty}(\mathbb R^{n})\,,\qquad P\in L^{\infty}(\mathbb R^{n}).
\ee
Of course, we continue assuming that the total population of ants is integrable: $u_o$ and $w_o$ belong to $ L^{1}(\mathbb R^{2})$, and we still denote
\[\displaystyle\int_{\mathbb R^2} (u_o+w_o)(x)\ud x=m_o.\]  
\subsection{Estimates for solutions of the heat equation}
The strategy for proving estimates for the  system \eqref{1000} is based on a sharp control of the gradient of the pheromone $\nabla p$ in terms of the ant population $w$: recall that for the system \eqref{2000} such an estimate was a direct consequence of the regularity analysis for elliptic equations, {\it {\`a la}} Agmon-Douglis-Nirenberg \cite{ADN}.
We substitute this argument with a direct analysis of the Duhamel formula for the heat equation.  To be more specific, let $\varphi$ be a solution of the problem
\begin{align}\label{HE}
\begin{split}
\partial_{t}\varphi - \Delta\varphi = f \quad &\text{in $(0,\infty)\times \mathbb{R}^{n}$},
\\[.3cm]
\varphi\Big|_{t=0} = \varphi_o \quad &\text{on $\mathbb{R}^{n}$}.
\end{split}
\end{align}
Then, $\varphi$ is given by the explicit formula
\begin{equation}\label{HEf}
\varphi(t,x)=\int_{\mathbb{R}^{n}}K(t,x-y)\varphi_o(y)\ud y + \int^{t}_{0}\int_{\mathbb{R}^{n}}K(t-s,x-y)f(s,y) \ud y\ud s\end{equation}
where
\[
K(t,x)=\displaystyle\frac{1}{(4\pi\,t)^{n/2}}\,e^{-\frac{|x|^{2}}{4t}},\]
as long as it makes sense.  Throughout this section the technical lemmas will be presented as general results valid for any space dimension $n\geq2$. 
\begin{lemma}\label{L1Linf}
Fix $q\in(n,\infty]$ and $\theta\in(0,2]$ such that
\begin{equation*}
\frac{\theta\, q}{2}>1\,,\qquad \frac{\theta\,q}{2-\theta}>n\,.
\end{equation*}
Additionally, assume that for some $T>0$
\begin{equation*}
\nabla\varphi_o\in L^{q}(\mathbb R^{n})\,, \qquad f \in L^{\infty}\big(0,T; L^{1}\cap L^{q\theta/2}(\mathbb{R}^{n})\big)\,.
\end{equation*}
Let $\varphi$ be a solution of the heat equation \eqref{HE}.  Then, for any $t\in(0,T]$
\begin{align}\label{estV}
\begin{split}
\|\nabla&\varphi(t,\cdot)\|_{q} \leq \|\nabla\varphi_o\|_{q} \\
&+C_{n}\Big(1+\frac{1}{t^{(n-q')/2q'}}\Big)\Big(1+\sup_{0\leq s \leq t}\|f(s,\cdot)\|_{1}\Big)\sup_{\tfrac{t}{2}\leq s \leq t}\|f(s,\cdot)\|^{\frac{\theta}{n}\frac{q(n-1)-n}{\theta q -2}}_{q\theta/2}\,.
\end{split}
\end{align}
The constant $C_{n}$ depends on the dimension $n$, $q$ and $\theta$.  Furthermore, the supremum in time can be estimated for any $T\geq1$ as
\begin{align}\label{estV1}
\begin{split}
\sup_{0\leq t \leq T}\|\nabla&\varphi(t,\cdot)\|_{q} \leq 2\|\nabla\varphi_o\|_{q}  \\
&+ C_{n}\Big(1+\sup_{0\leq s \leq t}\|f(s,\cdot)\|_{1}\Big)\sup_{0\leq s \leq T}\|f(s,\cdot)\|^{\frac{\theta}{n}\frac{q(n-1)-n}{\theta q -2}}_{q\theta/2}\,.
\end{split}
\end{align}
The constant $C_{n}$ depends on the dimension $n$, $q$ and $\theta$.
\end{lemma}
\proof
Using the explicit solution of the heat equation \eqref{HEf}, we obtain
\begin{align*}
\nabla \varphi(t,x)&=\int_{\mathbb{R}^{n}}\nabla_{x}K(t,x-y)\varphi_o(y) \ud y  \\[.3cm]
&\hspace{-.5cm} + \int^{t}_{0}\int_{\mathbb{R}^{n}}\nabla_{x}K(t-s,x-y) f(s,y)\,\ud y\ud s =: I_{1}(t,x) + I_{2}(t,x)\,.
\end{align*}
Estimate of the term  $I_{1}$  simply follows from  integration by parts and Young's inequality for convolutions
\begin{align}\label{estv1}
\begin{split}
\big\| I_{1}(t,\cdot) \big\|_{q} &=\Big\| \int_{\mathbb{R}^{n}}\nabla_{x}K(t,x-y) \varphi_o(y)\ud y \Big\|_{q}\\[.3cm]
&\hspace{-.5cm} = \Big\| \int_{\mathbb{R}^{n}}K(t,x-y)\nabla\varphi_o(y)\ud y\Big\|_{q}\leq \|K(t,\cdot)\|_{1}\|\nabla \varphi_o\|_{q} = \|\nabla \varphi_o\|_{q}.
\end{split}
\end{align}
Let us focus on the term $I_{2}$.  Fix $0<\varepsilon\leq t$ and consider the decomposition
$$
I_{2}(t,x) = \int^{t-\varepsilon}_{0}\int_{\mathbb{R}^{n}} ...\ud y\ud s+ \int^{t}_{t-\varepsilon}\int_{\mathbb{R}^{n}}...\ud y \ud s
=:I^{1}_{2}(t,x) + I^{2}_{2}(t,x)\,.
$$
In what follows, we will repeatedly make use of Young's  inequalities for convolution \cite[Th. IV.30]{Brez}
\begin{equation}\label{Yo}
\|f\star g\|_q\leq \|f\|_p\|g\|_r,
\qquad 1/p+1/r=1+1/q.
\end{equation}
For $I^{1}_{2}(t,x)$, it  yields
\begin{equation}\label{estv3}
\begin{array}{l}
\displaystyle\big\| I^{1}_{2}(t,\cdot) \big\|_{q}
\\[.3cm]
 = \Big\|-\frac{1}{2(4\pi)^{n/2}}\int^{t-\varepsilon}_{0}\frac{1}{(t-s)^{(n+1)/2}}\int_{\mathbb{R}^{n}}e^{-\frac{|x-y|^{2}}{4(t-s)}} \frac{(x-y)}{\sqrt{(t-s)}}\,f(s,y)\ud y\ud s\Big\|_{q}\\[.3cm]
\leq\displaystyle \frac{1}{2(4\pi)^{n/2}}\int^{t-\varepsilon}_{0}\frac{1}{(t-s)^{(n+1)/2}}\Big\|\int_{\mathbb{R}^{n}}e^{-\frac{|x-y|^{2}}{4(t-s)}} \frac{(x-y)}{\sqrt{(t-s)}}\,f(s,y)\,\ud y\Big\|_{q}\ud s\\[.3cm]
\leq\displaystyle \frac{1}{2(4\pi)^{n/2}}\int^{t-\varepsilon}_{0}\frac{1}{(t-s)^{(n+1)/2}}
\left(
\displaystyle\int_{\mathbb R^n}\Big|\displaystyle\frac{x}{\sqrt{t-s}}e^{-\frac{x^2}{4(t-s)}}
\Big|^q
\ud x
\right)^{1/q}
\,\big\|f(s,\cdot)\big\|_{1}\,\ud s\\[.3cm]
\leq \displaystyle\frac{\big\|e^{-|x|^2}x\big\|_{q}} { 4^{n/2q'}\pi^{n/2}}\frac{2q'}{n-q'}\frac{\sup_{0\leq s \leq t}\|f(s,\cdot)\|_{1}}{ \varepsilon^{(n-q')/2q'} }\,.
\end{array}\end{equation}
Note that the case $q=\infty$ is valid.  Next, we get (using \eqref{Yo} with $p=q\theta/2$ and $r=\sigma$)
\begin{equation}\label{estv4}\begin{array}{l}
\displaystyle\big\| I^{2}_{2}(t,\cdot) \big\|_{q} 
\\[.3cm]
\displaystyle = \Big\|-\frac{1}{2(4\pi)^{n/2}}\int^{t}_{t-\varepsilon}\frac{1}{(t-s)^{(n+1)/2}}\int_{\mathbb{R}^{n}}e^{-\frac{|x-y|^{2}}{4(t-s)}} \frac{(x-y)}{\sqrt{(t-s)}} f(s,y)\,\ud y\ud s\Big\|_{q}\\[.3cm]
\displaystyle\leq\frac{1}{2(4\pi)^{n/2}}\int^{t}_{t-\varepsilon}\frac{1}{(t-s)^{(n+1)/2}}\Big\|\int_{\mathbb{R}^{n}}e^{-\frac{|x-y|^{2}}{4(t-s)}} \frac{(x-y)}{\sqrt{(t-s)}} f(s,y)\,\ud y\Big\|_{q}\ud s\\[.3cm]
\displaystyle\leq \frac{1}{2(4\pi)^{n/2}}\int^{t}_{t-\varepsilon}\frac{1}{(t-s)^{(n+1)/2}}
\left(\displaystyle\int_{\mathbb R^n}\Big|\displaystyle\frac{x}{\sqrt{t-s}}e^{-\frac{x^2}{4(t-s)}}\Big|^\sigma\ud x\right)^{1/\sigma}
\,\big\|f(s,\cdot)\big\|_{q\theta/2} \ud s\\[.3cm]
\displaystyle\leq \frac{\big\|e^{-|x|^{2}}x\big\|_{\sigma'}} {4^{n/2\sigma'}\pi^{n/2}}\frac{2\sigma'}{\sigma'-n}\,\sup_{t-\varepsilon\leq s \leq t}\|f(s,\cdot)\|_{q\theta/2}\,\varepsilon^{(\sigma'-n)/2\sigma'}\,,
\end{array}
\end{equation}
where $\sigma'=\frac{q\theta}{2-\theta}>n$. (Note that the case $\theta=2$, that is $\sigma'=\infty$, is a valid choice.)  Thus, gathering \eqref{estv1}, \eqref{estv3} and \eqref{estv4} we are led to 
\begin{equation}\label{estv5}
\begin{array}{l}
\big\|\nabla\varphi(t,\cdot)\big\|_{q} 
= \big\|I_{1}(t,\cdot) + I_{2}(t,\cdot)\big\|_{q}\\[.3cm]
\qquad\displaystyle\leq \big\|\nabla\varphi_o\big\|_{q} + C_{n}\Bigg(\frac{\sup_{0\leq s \leq t}\|f(s,\cdot)\|_{1}}{\varepsilon^{(n-q')/2q'}} + \varepsilon^{(\sigma'-n)/2\sigma'}\sup_{ t-\varepsilon \leq s \leq t}\|f(s,\cdot)\|_{q\theta/2}\Bigg),
\end{array}
\end{equation}
with $C_n$ depending on $n, q,\theta$.  We shall choose $\varepsilon\in(0,\tfrac{t}{2}]$ defined by
\begin{equation*}
\varepsilon  =  \min\big\{1,\tfrac{t}{2}\big\} \Bigg(1 + \sup_{ \tfrac{t}{2} \leq s \leq t}\|f(s,\cdot)\|_{q\theta/2}\Bigg)^{-\frac{2}{n}\frac{\sigma'q'}{\sigma'-q'}}
\end{equation*}
to equalize the homogeneity.  Consequently, we obtain the estimate 
\begin{align*}
\|\nabla\varphi(t,\cdot)\|_{q} &\leq 
\|\nabla\varphi_o\|_{q} \\
&\hspace{-1cm}+C\Big(1+\frac{1}{t^{(n-q')/2q'}}\Big)\Big(1+\sup_{0\leq s \leq t}\|f(s,\cdot)\|_{1}\Big)\sup_{\tfrac{t}{2}\leq s \leq t}\|f(s,\cdot)\|^{\frac{\sigma'(n-q')}{n(\sigma'-q')}}_{q\theta/2}\,.
\end{align*}
where the constant $C>0$ depends only on the dimension $n$, on the exponent $q$ and $\theta$.  Estimate \eqref{estV} follows by noticing that
\begin{equation*}
\frac{\sigma'(n-q')}{n(\sigma'-q')} 
=\frac{q\theta}{2-\theta}\ \frac1n\ \frac{n-q/(q-1)}{q\theta/(2-\theta)-q/(q-1)}
=\frac{\theta}{n}\frac{q(n-1)-n}{\theta q -2}\,.
\end{equation*}
Now, for estimate \eqref{estV1} fix $T\geq1$ and gather \eqref{estv1}, \eqref{estv3} and \eqref{estv4} to obtain 
\begin{equation*}
\begin{array}{l}
\displaystyle\sup_{\varepsilon \leq t \leq T}\big\|\nabla\varphi(t,\cdot)\big\|_{q} 
= \displaystyle\sup_{\varepsilon \leq t \leq T}\big\|I_{1}(t,\cdot) + I_{2}(t,\cdot)\big\|_{q}\\[.3cm]
\displaystyle\leq \big\|\nabla\varphi_o\big\|_{q} + C_{n}\Bigg(\frac{\sup_{0\leq s \leq T}\|f(s,\cdot)\|_{1}}{\varepsilon^{(n-q')/2q'}} + \varepsilon^{(\sigma'-n)/2\sigma'}\sup_{ 0 \leq s \leq T}\|f(s,\cdot)\|_{q\theta/2}\Bigg),
\end{array}
\end{equation*}
with $C_n$ depending on $n, q,\theta$.  Note that $\varepsilon\in(0,T]$.  Similarly, when $t\in[0,\varepsilon]$ it follows that
\begin{align}\label{estv6}
\begin{split}
\sup_{0\leq t\leq \varepsilon}\|\nabla\varphi(t,\cdot)\|_{q} &= \sup_{0\leq t\leq \varepsilon }\|I_{1}(t,\cdot) + I_{2}(t,\cdot)\|_{q}\\[.3cm]
&\hspace{-1cm} \leq \|\nabla\varphi_o\|_{q} + C_{n}\,\varepsilon^{(\sigma'-n)/2\sigma'}\, \sup_{ 0 \leq s \leq T}\|f(s,\cdot)\|_{q\theta/2}\,.
\end{split}
\end{align}
Estimates \eqref{estv5} and \eqref{estv6} lead to
\begin{equation}\label{estv7}
\begin{array}{l}
\displaystyle\sup_{0\leq s\leq T}\|\nabla\varphi(s,\cdot)\|_{q} 
\\[.3cm]
\displaystyle\leq 2\|\nabla\varphi_o\|_{q} 
+ C_{n}\Bigg(\frac{\sup_{0\leq s \leq T}\|f(s,\cdot)\|_{1}}{\varepsilon^{(n-q')/2q'}} + \varepsilon^{(\sigma'-n)/2s'} \sup_{ 0 \leq s \leq T}\|f(s,\cdot)\|_{q\theta/2}\Bigg)\,.
\end{array}
\end{equation}
Again, we shall choose $\varepsilon\in(0,1]\subset(0,T]$ defined by
\begin{equation*}
\varepsilon  =  \Bigg(1 + \sup_{ 0 \leq s \leq T}\|f(s,\cdot)\|_{q\theta/2}\Bigg)^{-\frac{2}{n}\frac{\sigma'q'}{\sigma'-q'}}\end{equation*}
to equalize the homogeneity.  Consequently, we obtain the estimate 
$$\begin{array}{l}
\sup_{0\leq s\leq T}\|\nabla\varphi(s,\cdot)\|_{q}
\\[.3cm] \leq 
2\|\nabla\varphi_o\|_{q} 
+C\Big(1+\sup_{0\leq s \leq T}\|f(s,\cdot)\|_{1}\Big)\sup_{0\leq s \leq T}\|f(s,\cdot)\|^{\frac{\sigma'(n-q')}{n(\sigma'-q')}}_{q\theta/2}\,.
\end{array}$$
where the constant $C>0$ depends only on the dimension $n$, on the exponent $q$ and $\theta$.  This proves \eqref{estV1}.
\endproof
\begin{remark}
The same result applies to the damped equation
\begin{align}\label{HEd}
\begin{split}
\partial_{t}\varphi - \Delta\varphi = f -\delta \varphi\quad &\text{in $(0,\infty)\times \mathbb{R}^{n}$},
\\[.3cm]
\varphi\Big|_{t=0} = \varphi_o \quad &\text{on $\mathbb{R}^{n}$}.
\end{split}
\end{align}
Indeed, it suffices to apply  Lemma~\ref{L1Linf} to $e^{-\delta t}\varphi(t,x)$.
\end{remark}
\begin{lemma}\label{L2Linf}
Fix
\begin{equation*}
\begin{array}{c}
\varphi_o\in L^{\infty}(\mathbb{R}^{n})\,,\quad \nabla v\in L^{\infty}\big((0,T)\times\mathbb R^n\big)\,,\\ [5pt]
f\in L^{\infty}\big(0,T; L^{1} (\mathbb R^n)\big)\cap L^{\infty}\big((0,T)\times\mathbb R^n\big)\,.
\end{array}
\end{equation*} 
Let $\varphi$ be a nonnegative solution of the problem
\begin{align}\label{HEAdv}
\begin{split}
\partial_{t}\varphi - \Delta\varphi + \nabla\cdot(\varphi \nabla v) = f \;\; &\text{in}\;\;\mathbb{R}^{n}\times(0,\infty)\\[.3cm]
\varphi = \varphi_o \;\; &\text{on}\;\;\mathbb{R}^{n}\times\{t=0\}\,.
\end{split}
\end{align}
Then, we have for any $T\geq1$
\begin{align}\label{estW}
&\sup_{0\leq t\leq T}\|\varphi(t)\|_{\infty} \leq 2\|\varphi_{o}\|_{\infty} \\[.3cm] 
{}+C^1_{n}&\Big(\sup_{0\leq s \leq T}\|f_{+}(s)\|_{1}+ \sup_{0\leq s \leq T}\|\varphi(s)\|_{1}\Big)\|\nabla v\|^{n}_{\infty}  + C^2_{n}\sup_{0\leq s \leq T}\|f_{+}(s)\|_{\infty}^{\frac{n-1}{n+1}}\,.\nonumber
\end{align}
The constant $C^1_{n}$ is proportional to $\ln(T+1)$ in dimension $n=2$ and independent of time for $n\geq3$.
\end{lemma}

\proof
Fix $T \geq \varepsilon > 0$ and consider $t \in [\varepsilon,T]$.  We shall use the Duhamel formula for the heat equation 
 \eqref{HEf} with the source
term 
\begin{equation*}
\tilde{f} = f - \nabla\cdot(\varphi\,\nabla v).
\end{equation*}
We split the corresponding expression as $\varphi=\sum^{3}_{j=1} I_{j}$:  
let us concur here that $I_{1}$ is the term corresponding to the initial data, $I_{2}$ is the term corresponding to the source  $f$, and
 $I_{3}$ corresponds to the convection term $-\nabla\cdot(\varphi\,\nabla v)$.  For $I_{1}$, we immediately obtain
\begin{equation}\label{estI1}
0\leq I_{1}(t,x)=\int_{\mathbb{R}^{n}}K(t,x-y)\varphi_o(y)\ud y\leq \|\varphi_o\|_{\infty}\,.
\end{equation}
For $I_{2}$, we split the time integral as follows
\begin{equation*}
I_{2}(t,x) =\int^{t}_{0}\int_{\mathbb{R}^{n}}K(t-s,x-y)f(s,y)\ud y\ud s
=\underbrace{
 \int^{t-\varepsilon}_{0}\int_{\mathbb{R}^{n}}...\ud y\ud s}_{:=I^{1}_{2}} + 
 \underbrace{
 \int^{t}_{t-\varepsilon}\int_{\mathbb{R}^{n}}...\ud y\ud s}_{:=I^{2}_{2}}.
\end{equation*}
A direct computation
 shows that
\begin{align*}
|I^{1}_{2} | \leq C_{n}
\Bigg(
\begin{array}{cc}
\ln(t/\varepsilon) & \text{if}\;n=2\\[.3cm]
\varepsilon^{-\frac{n-2}{2}} & \text{if}\;n\geq 3
\end{array}
\Bigg)
\times\sup_{0\leq s\leq t-\varepsilon}\|f_{+}(s)\|_{1}
=: C_{n}\Phi(\varepsilon,t)\sup_{0\leq s\leq t-\varepsilon}\|f(s)\|_{1},
\end{align*}
while
\begin{align*}
|I^{2}_{2} | &\leq \int^{t}_{t-\varepsilon}\Bigg(\int_{\mathbb{R}^{n}}K(t-s,x-y)\ud y\Bigg)\ud s \sup_{t-\varepsilon\leq s \leq t}\|f_{+}(s)\|_{\infty}\\[.3cm]
&\leq
\int^{t}_{t-\varepsilon}1\ud s  \sup_{t-\varepsilon\leq s \leq t}\|f(s)\|_{\infty}=\varepsilon\sup_{t-\varepsilon\leq s \leq t}\|f_{+}(s)\|_{\infty}\,.
\end{align*}
As a consequence, we get
\begin{equation}\label{estI2}
I_{2}(t,x)\leq C_{n}\Bigg(\Phi(\varepsilon,t) \sup_{0\leq s \leq T}\|f_{+}(s)\|_{1} + \varepsilon\sup_{t - \varepsilon \leq s \leq T}\|f_{+}(s)\|_{\infty}\Bigg)\,.
\end{equation}
Additionally, integration by parts implies that
\begin{align*}
I_{3}(t,x) &= -\int^{t}_{0}\int_{\mathbb{R}^{n}}K(t-s,x-y)\nabla_{y}\cdot\big(\varphi(s,y) \nabla v(s,y)\big)\ud y\ud s\\[.3cm]
&= \int^{t}_{0}\int_{\mathbb{R}^{n}}\nabla_{y}K(t-s,x-y)\cdot\nabla v(s,y)\,\varphi(s,y)\ud y\ud s \\[.3cm]
&= \int^{t-\varepsilon}_{0}\int_{\mathbb{R}^{n}} ... \ud y\ud s+ \int^{t}_{t-\varepsilon}\int_{\mathbb{R}^{n}}...\ud y\ud s=:I^{1}_{3}+I^{2}_{3}\,.
\end{align*}
Similar  arguments lead to
\begin{align}\label{estI31}
\begin{split}
|I^{1}_{3}(t,x)| &= \left|\frac{1}{2(4\pi)^{n/2}}
\int^{t-\varepsilon}_{0}\frac{1}{(t-s)^{(n+1)/2}}\int_{\mathbb{R}^{n}}e^{-\frac{|x-y|^{2}}{4(t-s)}} \frac{(x-y)}{\sqrt{(t-s)}}\cdot\nabla v\,\varphi\,\ud y\ud s \right|\\[.3cm]
& \leq\frac{e^{-1/2}\|\nabla v\|_{\infty}} {\sqrt{2}(4\pi)^{n/2}(n-1)}\frac{\sup_{0\leq s \leq T}\|\varphi(s)\|_{1}}{\varepsilon^{(n-1)/2}}\,,
\end{split}
\end{align}
and
\begin{align}\label{estI32}
\begin{split}
I^{2}_{3}(t,x) &\leq \left|\frac{1}{2(4\pi)^{n/2}}\int^{t}_{t-\varepsilon}\frac{1}{(t-s)^{(n+1)/2}}\int_{\mathbb{R}^{n}}e^{-\frac{|x-y|^{2}}{4(t-s)}} \frac{(x-y)}{\sqrt{(t-s)}}\cdot\nabla v\,\varphi\,\ud y\ud s\right|
 \\[.3cm]
&\leq \frac{2\|e^{-|x|^{2}}x\|_{1}\|\nabla v\|_{\infty}} {\pi^{n/2}}\,\sqrt{\varepsilon} \sup_{t-\varepsilon\leq s \leq T}\|\varphi(s)\|_{\infty}.
\end{split}
\end{align}
Gathering estimates \eqref{estI31} and \eqref{estI32} it follows that
\begin{equation}\label{estI3}
I_{3}(t,x)\leq C_{n}\|\nabla v\|_{\infty}\Bigg(\frac{\sup_{0\leq s \leq T}\|\varphi(s)\|_{1}}{\varepsilon^{(n-1)/2}} + \sqrt{\varepsilon} \sup_{t-\varepsilon\leq s \leq T}\|\varphi(s)\|_{\infty}\Bigg)\,.
\end{equation}
Thus, combining \eqref{estI1}, \eqref{estI2} and \eqref{estI3} we obtain
\begin{align}\label{estw1}
\begin{split}
\sup_{\varepsilon \leq t \leq T}\|\varphi(t)&\|_{\infty} = \sup_{\varepsilon \leq t \leq T}\|I_{1}(t) + I_{2}(t) + I_{3}(t)\|_{\infty}\\[.3cm]
&\hspace{-1.2cm}\leq \|\varphi_o\|_{\infty} + C_{n}\Bigg(\Phi(\varepsilon,T) \sup_{0\leq s \leq T}\|f_{+}(s)\|_{1} + \varepsilon\sup_{ 0 \leq s \leq T}\|f_{+}(s)\|_{\infty}\Bigg) \\[.3cm]
&+ C_{n}\|\nabla v\|_{\infty}\Bigg(\frac{\sup_{0\leq s \leq T}\|\varphi(s)\|_{1}}{\varepsilon^{(n-1)/2}} + \sqrt{\varepsilon} \sup_{ 0 \leq s \leq T}\|\varphi(s)\|_{\infty}\Bigg).
\end{split}
\end{align}
Now assume $t \in [0,\varepsilon]$.  A simpler, yet analog, procedure shows that
\begin{align}\label{estw2}
\begin{split}
\sup_{0\leq t\leq \varepsilon}\|& \varphi(t)\|_{\infty} = \sup_{0\leq t\leq \varepsilon }\|I_{1}(t) + I_{2}(t) + I_{3}(t)\|_{\infty}\\[.3cm]
& \leq \|\varphi_o\|_{\infty} + \varepsilon\,\sup_{ 0 \leq s \leq T}\|f_{+}(s)\|_{\infty} + \|\nabla v\|_{\infty}\,\sqrt{\varepsilon}\, \sup_{ 0 \leq s \leq T}\|\varphi(s)\|_{\infty}\,.
\end{split}
\end{align}
Noticing that
\begin{equation*}
\sup_{0\leq t\leq T}\|\varphi(t)\|_{\infty} \leq \sup_{0\leq t\leq \varepsilon}\|\varphi(t)\|_{\infty} + \sup_{\varepsilon\leq t \leq T}\|\varphi(t)\|_{\infty}\,,
\end{equation*}
we can add inequalities \eqref{estw1} and \eqref{estw2} to obtain
\begin{align}\label{estw3}
\sup_{0\leq t\leq T}&\|\varphi(t)\|_{\infty} \leq 2\|\varphi_o\|_{\infty} + C_{n}\Bigg(\Phi(\varepsilon,T)\sup_{0\leq s \leq T}\|f_{+}(s)\|_{1} + \varepsilon\sup_{ 0 \leq s \leq T}\|f_{+}(s)\|_{\infty}\Bigg)\nonumber\\[.3cm]
&\qquad\qquad+ C_{n}\|\nabla v\|_{\infty}\Bigg(\frac{\sup_{0\leq s \leq T}\|\varphi(s)\|_{1}}{\varepsilon^{(n-1)/2}} + \sqrt{\varepsilon} \sup_{ 0 \leq s \leq T}\|\varphi(s)\|_{\infty}\Bigg).
\end{align}
Let us choose in all dimensions
\begin{equation*}
\varepsilon  =  \min\big\{1,\,\delta^{2},\, T\big\} \Bigg(1 + \sup_{ 0 \leq s \leq T}\|f_{+}(s)\|_{\infty}\Bigg)^{-\frac{2}{n+1}}
\end{equation*}
with $1/\delta := 2C_{n}\|\nabla v\|_{\infty}$, to discover that
\begin{align*}
&\sup_{0\leq t\leq T}\|\varphi(t)\|_{\infty} \leq 2\|\varphi_{o}\|_{\infty} \\[.3cm] 
{}+C^1_{n}&\Big(\sup_{0\leq s \leq T}\|f_{+}(s)\|_{1}+ \sup_{0\leq s \leq T}\|\varphi(s)\|_{1}\Big)\|\nabla v\|^{n}_{\infty}  + C^2_{n}\sup_{0\leq s \leq T}\|f_{+}(s)\|_{\infty}^{\frac{n-1}{n+1}}\,.\nonumber
\end{align*}
The constant $C^{1}_{n}$ inherits the dependence of $T$ from $\Phi(\varepsilon,T)$.  Thus, for any $T\geq 1$, we get
\begin{equation*}
C^1_{n}=\Bigg\{
\begin{array}{cc}
\tilde {C}^1_{n}\big(\ln(T+1)\big) & \text{if}\;n=2\,, \\[.3cm]
\text{Independent of }\;T  & \text{if}\; n\geq 3.
\end{array}
\end{equation*}
This proves the lemma.
\endproof
\subsection{Analysis of the system \eqref{1000} in dimension $n=2$}\label{finalsubsect}
\subsubsection{From $L^{1}$ to $L^{\gamma}$ integrability}
\begin{proposition}\label{SPDInt}
Fix a time $T>0$ and let $(u,w)$ be a classical nonnegative solution of the SPD system \eqref{1000} in the interval $[0,T]$.  Then, for any $\gamma>1$ there exists an explicitly computable exponent $\beta:=\beta(\gamma)$ such that
\begin{align*}
\|u(t)\|_{\gamma} + \|w(t)\|_{\gamma}  \le C\big(m_o,\gamma\big)\,\Big(1+\frac{1}{t^{\beta}}\Big)\,,\quad t\in(0,T]\,,
\end{align*}
with constant $C$ depending additionally on\eqref{HypH}-\eqref{HSPD} and $\|\nabla p_o\|_{L^{2(\gamma+1)}}$ but  independent of $T$.  If additionally $(u_o, w_o)\in L^{\gamma}\times L^{\gamma^{+}}$ for some $\gamma>\sqrt{2}$, it follows that
\begin{align*}
\sup_{0\leq t \leq T}\Big(\|u(t)\|_{\gamma} + \|w(t)\|_{\gamma^{+}}\Big)  \le C\big(m_o,\gamma,\|u_o\|_{\gamma},\|w_o\|_{\gamma^{+}}\big)\,,
\end{align*}
where, as before, the constant $C$ depends additionally on\eqref{HypH}-\eqref{HSPD} and $\|\nabla p_o\|_{L^{2(\gamma+1)}}$ but it is independent of $T$.
\end{proposition}
\proof 
For the population $u$, we obtain the following estimate
\begin{align*}
\frac{\ud}{\ud t}\int_{\mathbb{R}^{n}} & u^\gamma   \ud x + 4 \frac{\gamma-1}{\gamma}\int_{\mathbb{R}^{n}} |\nabla u^{\gamma/2}|^2  \ud x\\[.3cm]
& \quad \leq  (\gamma-1)\int_{\mathbb{R}^{n}} \nabla p \cdot \nabla (u^{\gamma})   \ud x + \gamma \int_{\mathbb{R}^{n}} N w u^{\gamma-1}  \ud x\,.
\end{align*}
With Cauchy-Schwarz and Young  inequalities, we get
\[\begin{array}{l}
\displaystyle(\gamma-1)\int_{\mathbb{R}^{n}} \nabla p \cdot \nabla (u^{\gamma})   \ud x 
= 2(\gamma-1)\int_{\mathbb{R}^{n}} \nabla p \cdot \nabla (u^{\gamma/2})u^{\gamma/2}   \ud x\\[.3cm]
 \displaystyle
 \qquad\leq \gamma(\gamma-1)\int_{\mathbb{R}^{n}}\big|\nabla p\big|^{2}u^{\gamma} \ud x + \frac{\gamma-1}{\gamma}\int_{\mathbb{R}^{n}}\big|\nabla (u^{\gamma/2})\big|^{2}\ud x\,,
\end{array}\]
so that
\begin{align*}
\frac{\ud}{\ud t} \int_{\mathbb{R}^{n}} u^\gamma   \ud x
+ & 3 \frac{\gamma-1}{\gamma}\int_{\mathbb{R}^{n}} |\nabla u^{\gamma/2}|^2  \ud x
\nonumber\\[.3cm]
&\leq  \gamma \int_{\mathbb{R}^{n}}  N w u^{\gamma-1}   \ud x +  \gamma(\gamma-1)\int_{\mathbb{R}^{n}}\big|\nabla p\big|^{2}u^{\gamma} \ud x.
\end{align*}
We make use of the H\"older inequality with conjugate exponents $\gamma$ and $\gamma'=\gamma/(\gamma-1)$ in the first integral of the right hand side
and with the pair $\gamma+1$, $(\gamma+1)/\gamma$ in the second.
Combined with the convexity inequality $ab\leq a^p/p+b^{p'}/p'$, we find, for any $\varepsilon>0$,
\begin{align*}
\frac{\ud}{\ud t} \int_{\mathbb{R}^{n}} u^\gamma   \ud x
+ & 3 \frac{\gamma-1}{\gamma}\int_{\mathbb{R}^{n}} |\nabla u^{\gamma/2}|^2  \ud x
\\[.3cm]
&\hspace{-1cm}\leq
\gamma \|N\|_{\infty} \int_{\mathbb{R}^{n}} w^{\gamma}  \ud x +  \gamma\,\|N\|_{\infty} \int_{\mathbb{R}^{n}} u^{\gamma}  \ud x  +  {}\nonumber\\[.3cm]
&\hspace{1.2cm}  \frac{\gamma}{\varepsilon^{\gamma+1}}\int_{\mathbb{R}^{n}}   \big|\nabla p\big|^{2(\gamma+1)}  \ud x + \gamma^{2}\varepsilon^{\frac{\gamma+1}{\gamma}}\int_{\mathbb{R}^{n}} u^{\gamma+1}  \ud x\,.\nonumber
\end{align*}
In order to control $\nabla p$ we use estimate \eqref{estV} in Lemma \ref{L1Linf} with $n=2$, $q=2(\gamma+1)$, $\theta=\frac{\gamma}{\gamma+1}$, and $f=P\,w$, to conclude that
\begin{equation*}
\begin{array}{l}\displaystyle
\int_{\mathbb{R}^{n}}\big|\nabla p(t,x)\big|^{2(\gamma+1)}\ud x
\\[.3cm]
\displaystyle
\leq \|\nabla p_o\|_{2(\gamma+1)}^{2(\gamma+1)} 
+C(m_o)\Big(1+ \frac{1}{t^{\gamma}}\Big)\sup_{\tfrac{t}{2}\leq s\leq t}\Big(\int_{\mathbb{R}^{n}}\big|w(s,x)\big|^{\gamma}\ud x\Big)^{\frac{\gamma}{\gamma-1}}
\end{array}\end{equation*}
holds for any $\gamma>\sqrt{2}$.  The constant $C(m_o)$ depends on \eqref{HypH} and \eqref{HSPD}.  This leads to
\begin{align}\label{SDEP1}
\begin{split}
\frac{\ud}{\ud t} &\int_{\mathbb{R}^{n}} u^\gamma(t,x)    \ud x
+ 3 \frac{\gamma-1}{\gamma}\int_{\mathbb{R}^{n}} |\nabla u^{\gamma/2}(t,x)|^2  \ud x
\\[.3cm]
& \leq \gamma\|N\|_{\infty} \int_{\mathbb{R}^{n}} w^{\gamma}(t,x)  \ud x + {}\\[.3cm]
 &\hspace{-.2cm}\gamma\,\|N\|_{\infty} \int_{\mathbb{R}^{n}} u^{\gamma}(t,x)  \ud x 
 + \frac{C(m_o)\,\gamma}{\varepsilon^{\gamma+1}}\Big(1+\frac{1}{t^{\gamma}}\Big)\sup_{\tfrac{t}{2}\leq s\leq t}
 \Big(\int_{\mathbb{R}^{n}}w^{\gamma}(s,x)\ud x\Big)^{\frac{\gamma}{\gamma-1}} \\[.3cm]
&\qquad+ \gamma^{2}\varepsilon^{\frac{\gamma+1}{\gamma}}\int_{\mathbb{R}^{n}} u^{\gamma+1}(t,x)  \ud x\,
+ \frac{\gamma}{\varepsilon^{\gamma+1}} \|\nabla p_o\|^{2(\gamma+1)}_{2(\gamma+1)}\,.
\end{split}
\end{align} 
We also have the following  estimate for $w$
\begin{align}\label{SDEP2}
\frac{\ud}{\ud t} \int_{\mathbb{R}^{n}} & w^\alpha(t,x)  \ud x 
+ 4 \frac{\alpha-1}{\alpha}\int_{\mathbb{R}^{n}} |\nabla w^{\alpha/2}(t,x)|^2  \ud x\\[.3cm]
&\leq \|c\|_{\infty} \int_{\mathbb{R}^{n}} u^{\alpha}(t,x)  \ud x 
+ \alpha\big(\|c\|_{\infty}+\|\Delta v\|_{\infty}\big) \int_{\mathbb{R}^{n}} w^{\alpha}(t,x)  \ud x\,.\nonumber
\end{align}
In other to control the right side with the left side we are forced to select $\gamma<\alpha<\gamma+1$.  
We shall also use Gagliardo-Nirenberg-Sobolev's inequality:  in the whole space $\mathbb R^2$, inequality \eqref{4500} reduces to $\int \xi^{\alpha+1}\ud x\leq C\int \xi\ud x\int \nabla(\xi^{\alpha/2})\ud x$.  Thus, adding inequalities \eqref{SDEP1} and \eqref{SDEP2} we obtain, after similar computations to those of the FPD system, the following estimate which holds for any $t\in(0,T]$
\begin{align*}
\frac{\ud}{\ud t}Z(t) + C(m_o)& Z^{\alpha'}(t)\leq C\big(1+\|\nabla p_o\|^{2(\gamma+1)}_{2(\gamma+1)} \big)+ C(m_o)\Big(1+\frac{1}{t^{\gamma}}\Big)\sup_{\tfrac{t}{2}\leq s \leq t}\|w\|^{\gamma\gamma'}_{\gamma}\\
&\hspace{-1cm}\leq C\big(1+\|\nabla p_o\|^{2(\gamma+1)}_{2(\gamma+1)} \big)+ C(m_o)\Big(1+\frac{1}{t^{\gamma}}\Big)\sup_{\tfrac{t}{2}\leq s \leq t}\|w\|^{\gamma\alpha'}_{\alpha} \\
&\leq C\big(1+\|\nabla p_o\|^{2(\gamma+1)}_{2(\gamma+1)} \big)+ C(m_o)\Big(1+\frac{1}{t^{\gamma}}\Big)\sup_{\tfrac{t}{2}\leq s \leq t}Z^{\frac{\gamma}{\alpha}\alpha'}(t)\,.
\end{align*}
where
\begin{equation*}
Z(t):=\int_{\mathbb{R}^{n}}u^{\gamma}(t)\ud x + \int_{\mathbb{R}^{n}}w^{\alpha}(t)\ud x\,.
\end{equation*}
Now apply the comparison Lemma \ref{A2} to obtain for any $\gamma>\sqrt{2}$
\begin{equation}\label{SDEGenL1}
Z(t)\leq C(m_o,\gamma)\Big(1+\frac{1}{t^{\frac{\alpha-1}{\alpha-\gamma}\gamma}}\Big)\,,\qquad 0<t\leq T\,.
\end{equation}
The constant $C$ depends additionally on\eqref{HypH}-\eqref{HSPD} and $\|\nabla p_o\|_{L^{2(\gamma+1)}}$ but it is independent of $T>0$.  The case $\gamma\in(1,\sqrt{2}]$ follows by Lebesgue's interpolation between estimate \eqref{SDEGenL1} and the mass conservation.

Finally, uniform propagation of the $L^{\gamma}$ and $L^{\gamma^{+}}$ norms of $u$ and $w$ follows using previous computations with the estimate \eqref{estV1} instead of estimate \eqref{estV} which give us the bound
\begin{equation*}
\frac{\ud}{\ud t}Z(t) + C(m_o) Z^{\alpha'}(t)\leq C\big(1+\|\nabla p_o\|^{2(\gamma+1)}_{2(\gamma+1)} \big)+ C(m_o)\sup_{0\leq s \leq T}Z^{\frac{\gamma}{\alpha}\alpha'}(s)\,,
\end{equation*}
valid for any $t\in[0,T]$ with $T\geq1$ and $\sqrt{2}<\gamma<\alpha<\gamma+1$.  The comparison Lemma \ref{A1} gives
\begin{align*}
&\sup_{0\leq s \leq T}Z(s)\leq\\
&\hspace{.5cm}\max\Big\{Z(0), \Big(C\big(1+\|\nabla p_o\|^{2(\gamma+1)}_{2(\gamma+1)} \big)+ C(m_o)\sup_{0\leq s \leq T}Z^{\frac{\gamma}{\alpha}\alpha'}(s)\Big)^{\frac{1}{\alpha'}}\Big\}\,,
\end{align*}
which implies, since $\tfrac{\gamma}{\alpha}<1$, that $\sup_{0\leq s \leq T}Z(s)$ is finite and uniform in $T$. 
\endproof
\subsubsection{From $L^{\gamma}$ to $L^{\infty}$ integrability}
\begin{proposition}\label{SPDP1}
Let  $T>0$. Consider initial data $\big(u_o, w_o, \nabla p_o\big)\in L^{3} \times L^{3^{+}}\times L^{8}$ and, let $(u,w)$ be a classical nonnegative solution of the  system \eqref{1000} in the interval $[0,T]$.  Then, the following $L^{\infty}$-estimate holds
\begin{align}\label{SPDBound}
\|u(t)\|_{\infty}+ \|w(t)\|_{\infty} &\leq  C(m_o)\Big(1+ \frac{1}{\sqrt{t}} \Big)\,, \quad t\in(0,T]\,,
\end{align}
where the constant $C$ depends on \eqref{HypH}, \eqref{HSPD} and the initial data, but, it is independent of $T>0$.  In particular, for any $\gamma\in[1,\infty]$ it follows that
\begin{align}\label{SPDinterpolation}
\|u(t)\|_{\gamma}+ \|w(t)\|_{\gamma} &\leq  C(m_o)\Big(1+ \frac{1}{t^{\frac{1}{2\gamma'}}} \Big)\,, \quad t\in(0,T]\,.
\end{align}
Furthermore, estimate \eqref{SPDBound} can be upgraded by adding the dependence on the $L^{\infty}$-norms of the initial data in the constant,
\begin{equation}\label{SPDBound1}
\sup_{0\leq s\leq T}\|w(s)\|_{\infty} + \sup_{0\leq s \leq T} \|u(t)\|_{\infty} \leq C\big(m_o,\|w_o\|_{\infty},\|u_o\|_{\infty}\big)\,,
\end{equation}
The constant depend on \eqref{HypH}, \eqref{HSPD} and $\|\nabla p_o\|_{\infty}$, but, it is independent of the time.
\end{proposition}
\proof
The proof of estimate \eqref{SPDBound} is a direct consequence of De Giorgi level set technique presented in the proof of Proposition \ref{P1} and Proposition \ref{SPDInt}.  The integrability of the initial data assures that the delicate nonlinear term $u\,\nabla p$ will remain uniformly bounded as required in these arguments. 
\\

Interestingly, the uniform bound \eqref{SPDBound1} is not straightforward to prove.  The main reason is the lack of knowledge about the regularity of $\Delta p$, which was a consequence of  elliptic regularity for  the FPD system.  
This is where we make use of Lemma~\ref{L2Linf}.  Indeed, let us apply Lemma \ref{L2Linf} to the population $u$, that is, $\varphi = u$, $f_{+}=w\,N$, and recalling that conservation of mass implies
\begin{equation*}
\sup_{0\leq s \leq T}\|w(s)\|_{1}+ \sup_{0\leq s \leq T}\|u(s)\|_{1} \leq 2m_o\,.
\end{equation*}
Thus, we get
\begin{align}\label{SPDLinf1}
\begin{split}
\sup_{0\leq s \leq T}&\|u(s)\|_{\infty} \leq 2\|u_o\|_{\infty} \\
&{}+C^{1}_{n}\sup_{0\leq s \leq T}\|\nabla p(s)\|^{n}_{\infty} + C^{2}_{n}\sup_{0\leq s \leq T}\|w(s)\|^{\frac{n-1}{n+1}}_{\infty}\,.
\end{split}
\end{align}
Now, apply Lemma \ref{L1Linf}, with $q=\infty$ and $\theta=2$, to the pheromone equation to obtain
\begin{equation}\label{SPDLinf2}
\sup_{0\leq s \leq T}\|\nabla p(s)\|_{\infty} \leq 2\|\nabla p_o\|_{\infty} + C^{3}_{n}\sup_{0\leq s \leq T}\|w(s)\|^{\frac{n-1}{n}}_{\infty}\,.
\end{equation}
Noticing that $c\leq c_o$, we may apply Lemma \ref{L2Linf} to the population $w$ as well to conclude that
\begin{equation}\label{SPDLinf3}
\sup_{0\leq s \leq T}\|w(s)\|_{\infty} \leq 2\|w_o\|_{\infty} +
C^{4}_{n}\|\nabla v\|^{n}_{\infty} + C^{5}_{n}\sup_{0\leq s \leq T}\|u(s)\|^{\frac{n-1}{n+1}}_{\infty}\,.
\end{equation}
In addition to the mass, the constants depend on the $L^{\infty}$-norms of the parameters and grow logarithmically in time.  Plugging successively \eqref{SPDLinf3} in \eqref{SPDLinf2} and the result in \eqref{SPDLinf1} the estimate for the population $u$ reduces to
\begin{equation}\label{SPDLinf4}
\sup_{0\leq t\leq T}\|u(t)\|_{\infty} \leq C_o + C_{1} \sup_{0\leq s\leq T}\|u(s)\|^{\frac{(n-1)^{2}}{n+1} }_{\infty} + C_{2} \sup_{0\leq s\leq T}\|u(s)\|^{\frac{(n-1)^{2}}{(n+1)^{2}} }_{\infty}
\end{equation}
where $C_o$ depends on the initial data and $\|\nabla v\|_{\infty}$.  Since $\frac{1}{3}=\frac{(n-1)^{2}}{n+1}<1$ the result follows for short $T>0$ with constants growing at most $\ln(T)$.  For large $T>0$ the result follows after simple interpolation with estimate \eqref{SPDBound}.
\endproof

\endproof

\begin{remark}\label{Ren=3}
The maximal exponent in the right side of \eqref{SPDLinf4} $\frac{(n-1)^{2}}{n+1}\geq1$ for $n\geq 3$.  This suggest, for these dimensions, a finite in time Dirac concentration of mass similar to that of the Keller-Segel model in two or more dimensions.  Similarly, such concentration could be avoided by smallness conditions on the initial data and model parameters.
\end{remark}

\appendix
\section{A useful comparison lemma}
\begin{lemma}[ODE comparison]\label{A1} Assume $Y$ and $X$ are absolutely continuous functions in $[0,T]$ and such that
\begin{align}
\label{supersub}
\begin{split}
Y'(t)+ a\,Y^{\alpha}(t)&\geq b + \delta + c\Big(1+\frac{1}{t^{\gamma}}\Big)\sup_{\tau \leq s \leq t}Y^{\alpha_o}(s)\\[0.1cm]
X'(t)+ a\,X^{\alpha}(t)&\leq b + c\Big(1+\frac{1}{t^{\gamma}}\Big)\sup_{\tau \leq s \leq t}X^{\alpha_o}(s) \,,
\end{split}
\end{align}
with  $b\geq0$, $c\geq0$, $a>0$, $\delta>0$, $\alpha>\alpha_o\geq0$, $\gamma\geq0$ and $t\geq \tau \geq 0$.  If $Y(0) > X(0)$ then $Y\geq X$ in $[0,T]$.  In particular, if $\gamma=0$
\begin{equation}\label{subest}
\sup_{t\in[0,T]}X(t) \leq \max\{X(0), C\}\,,
\end{equation}
where the constant $C>0$ depends on all parameters but $\tau$, $\delta$ and $T$.
\end{lemma}
\begin{proof}
Define
\begin{equation*}
t_o:=\sup\big\{ t: Y(s)\geq X(s),\;s\in[0,t] \big\}\,.
\end{equation*}
Note that $t_{o}>0$ since $Y(0)>X(0)$.  Let us argue by contradiction assuming that there exists $t_{1}\in(0,T]$ such that $X(t_{1})>Y(t_{1})$.  Since $X$ and $Y$ are continuous $t_o\in(0,T)$ and $X(t_o)=Y(t_o)$.  Also, there exists interval $I=(t_{o},t_{o}+\varepsilon)$ such that $X(t)>Y(t)$ for any $t\in I$.  Thus, the fundamental theorem of calculus implies
\begin{equation*}
\int^{t}_{t_o}X'(s)\ud s = X(t) - X(t_o) > Y(t) - Y(t_o) = \int^{t}_{t_o}Y'(s)\ud s\,,\quad t\in I\,.
\end{equation*}
Therefore, there exists $t_{\star}\in I$ sufficiently close to $t_o$ such that: (1) $X$ and $Y$ are differentiable at $t_{\star}$ with $X'(t_{\star})>Y'(t_{\star})$, and (2) $c\big(1+\frac{1}{t^{\gamma}_{*}}\big)\big(\sup_{\tau \leq s \leq t_{*}}Y^{\alpha_o}(s) - \sup_{\tau \leq s \leq t_{*}}X^{\alpha_o}(s)\big) \geq -\delta$.  Thus, using \eqref{supersub}
\begin{align*}
0 > Y'(t_{\star}) & - X'(t_{\star}) \geq a\big( X^{\alpha}(t_{\star}) - Y^{\alpha}(t_{\star}) \big)  \\
&+\delta + c\Big(1+\frac{1}{t^{\gamma}_{*}}\Big)\Big(\sup_{\tau \leq s \leq t_{*}}Y^{\alpha_o}(s) - \sup_{\tau \leq s \leq t_{*}}X^{\alpha_o}(s)\Big) \geq 0\,.
\end{align*}
This contradicts the existence of $t_{1}$.  Finally, the estimate \eqref{subest} follows by taking $Y:=Y_{\delta}$ as the constant function in $[0,T]$ given by $\max\{X(0)+\delta,C_{\delta}\}$ where
\begin{equation*}
C_{\delta} = \frac{2}{a}\Big(b+\delta+\frac{1}{\varepsilon^{(\alpha/\alpha_o)'}}\Big)\,,\quad \varepsilon^{\alpha/\alpha_o} = \frac{a}{4c} \,.
\end{equation*}
Then, $X(t) \leq Y_{\delta}$ for any $\delta>0$.  The result follows by sending $\delta\rightarrow0$.
\end{proof}
\begin{corollary}\label{A2}
Assume $X$ be an absolutely continuous function in $[0,T]$ such that
\begin{equation*}
X' + a\,X^{\alpha}\leq b + c\Big(1+\frac{1}{t^{\gamma}}\Big)\sup_{\tfrac{t}{2} \leq s \leq t}X^{\alpha_o}(s)\,,
\end{equation*}
with  $b\geq0$, $c\geq0$, $a>0$, $\alpha>\alpha_o\geq0$ and $\gamma\geq0$.  Then, 
\begin{equation*}
X(t)\leq C\Big(1+\frac{1}{t^{\beta}} \Big)\,,\quad \beta=\max\big\{\tfrac{1}{\alpha-1},\tfrac{\gamma}{\alpha-\alpha_o}\big\}\,.
\end{equation*}
The constant $C>0$ depends on all parameters but it is independent of $T$.  
\end{corollary}
\begin{proof}
First note that the function
\begin{equation*}
Y(t)=C\Big(1+\frac{1}{t^{\beta}} \Big)
\end{equation*}
satisfies $Y'(t)+ a\,Y^{\alpha}(t)\geq b + \delta + \big(1+\frac{1}{t^{\gamma}}\big)\sup_{\frac{t}{2} \leq s \leq t}Y^{\alpha_o}(s)$ for $C$ sufficiently large depending on all parameters.  Second, since $X$ is bounded on $[0,T]$, there exists sufficiently small $t_o>0$ such that $Y(s)>\sup_{t\in[0,T]} X(t)$ for $s\in(0,t_o)$.  Applying Lemma \ref{A1} in the interval $[s,T]$ it follows that $Y\geq X$ in $[s,T]$.  The result follows since $s$ can be taken arbitrarily small.
\end{proof}

\section*{Acknowledgements}

We acknowledge support form the Brazilian-French Network in Mathematics, which has made
possible a visit in Rio de Janeiro where a large part of this work has been done. Th. G. thanks
both the Math. Dept. at PUC and IMPA for their warm welcome. P.A.~was partially supported by FAPERJ grant no. APQ1 - 111.400/2014  and CNPq grant no. Universal - 442960/2014-0.
%¥Bib

\bibliographystyle{plain}
\bibliography{./AAG}
\end{document}